\documentclass[11pt,reqno]{amsart}
\usepackage{enumitem}
\usepackage{url}

\usepackage[utf8]{inputenc}

\usepackage{graphicx}
\usepackage{dcolumn}
\usepackage{enumitem}
\usepackage{bm}
\usepackage[breaklinks=true,colorlinks=true,linkcolor=blue,urlcolor=blue,citecolor=blue]{hyperref}
\usepackage{amsfonts,amsthm,amsmath,amssymb}
\usepackage{chngcntr}

\counterwithout{figure}{section}
\counterwithout{figure}{subsection}

\theoremstyle{plain}
\newtheorem{theorem}{\bf Theorem}[section]
\newtheorem{proposition}{\bf Proposition}[section]
\newtheorem{lemma} {\bf Lemma}[section]
\theoremstyle{definition}
\newtheorem{definition}{Definition}[section]
\theoremstyle{remark}
\newtheorem*{remark}{Remark}
\newtheorem{example}{Example}[section]

\newcommand{\QQ}{\mathbb{Q}}
\newcommand{\frakC}{\mathfrak{C}}

\newcommand{\D}{\mathbf{D}}
\newcommand{\C}{\mathbf{C}}
\newcommand{\R}{\mathbf{R}}
\newcommand{\re}{\operatorname{Re}}
\newcommand{\Cauchy}{\mathcal{C}}
\newcommand{\T}{\mathbf{T}}
\newcommand{\1}{\mathbf{1}}
\newcommand{\calC}{\mathcal{C}}
\newcommand{\calW}{\mathcal{W}}
\newcommand{\calS}{\mathcal{S}}

\newcommand{\calI}{\mathcal{I}}
\newcommand{\calF}{\mathcal{F}}
\newcommand{\fii}{{\varphi}}
\newcommand{\supp}{\operatorname{supp}}
\def\d{\partial}
\def\dbar{\bar{\partial}}

\usepackage{float}

\calclayout

\newcommand{\CC}{\mathbf{C}}
\newcommand{\RR}{\mathbf{R}}

\begin{document}
\def\phi{\varphi}


\title{On the uniqueness problem for quadrature domains}

\author{Yacin Ameur}
\address{Yacin Ameur\\
Department of Mathematics\\
Lund University\\
22100 Lund, Sweden}
\email{ Yacin.Ameur@math.lu.se}

\author{Martin Helmer}%
\address{Martin Helmer, \href{https://orcid.org/0000-0002-9170-8295}{https://orcid.org/0000-0002-9170-8295}\\
Mathematical Sciences Institute\\
The Australian National University\\
Canberra, Australia}
 \email{martin.helmer@anu.edu.au}

\author{Felix Tellander}
\address{Felix Tellander, \href{https://orcid.org/0000-0001-6418-8047}{https://orcid.org/0000-0001-6418-8047}\\
 Department of Astronomy and Theoretical Physics\\ Lund University
 \\ SE-223 62, Lund, Sweden
}%
\email{felix@tellander.se}

\


\begin{abstract}
We study questions of existence and uniqueness of quadrature domains using computational tools from real algebraic geometry.
 These problems are transformed into questions about the number of solutions to an associated real semi-algebraic system, which is analyzed
 using the method of
  real comprehensive triangular decomposition.
\end{abstract}

\maketitle

\section{\label{sec:level1}Introduction}

This note is the result of investigations into an open uniqueness question
for quadrature domains in the complex plane $\C$, which appears in papers such as \cite{G90,GuS05,Za87}. After describing the problem and reviewing some known results, we will suggest and explore an approach based on methods from real algebraic geometry and symbolic computation.

To get started, it is convenient to fix some notation that will be used \textit{throughout}.

\subsubsection*{General notation}

By a ``domain'' $\Omega$ we mean an open and connected subset of $\C$; we write
$\overline{\Omega}$ for its closure, $\d\Omega$ for its boundary, and $\Omega^e=\C\setminus \overline{\Omega}$ for its exterior.
A bounded domain $\Omega$ is said to be ``solid'' if
$\Omega^e$ is connected and $\d\Omega=\d(\Omega^e)$. Note that what we call solid domains are also referred to as ``Carath\'eodory domains" in some sources.

We write ``$dA$'' for the normalized Lebesgue measure in the plane $dA(z)=\frac 1 \pi\, dx\, dy$ (so the unit disc has area $1$).

We denote by $L^1(\Omega)$ the usual $L^1$-space of functions on $\Omega$ that are integrable with respect to $dA$, and we write $AL^1(\Omega)$, $HL^1(\Omega)$, and $SL^1(\Omega)$ for the subsets of $L^1(\Omega)$ consisting of all analytic, harmonic, and subharmonic functions, respectively.

Standard sets: $D(a,r)=\{z\in\C;\,|z-a|<r\}$, $\D=D(0,1)$, $\T=\d \D$.

Differential operators: $\d=\frac 1 2(\d_x-i\d_y)$, $\dbar=\frac 1 2 (\d_x+i\d_y)$, $\Delta=4\d\dbar=\d_x^2+\d_y^2.$
\\\\
Given an open set $\Omega\subset\C$, a subspace (or cone) $\calF\subset L^1(\Omega)$, and a linear functional $\mu:\calF\to\C$,
we consider quadrature identities of the form
\begin{equation}\label{qi}\int_\Omega f\, dA=\mu(f),\qquad f\in\calF.\end{equation}
The linear functional $\mu$ will always be a fixed measure or distribution of appropriate type with \textit{compact support} in $\Omega$
(and defined on the appropriate test-class $\calF$). Since we will frequently take $\mu$ to be a combination of point-evaluations, we stress
that statements like $\calF\subset L^1(\Omega)$ should not be taken literally, but rather in terms of the natural injective maps $\calF\to L^1(\Omega)$.

If the above conditions hold, we say that $\Omega$ is a \textit{quadrature domain} (or ``q.d.'') with data $(\mu,\calF)$,
 and we write $\Omega \in Q(\mu,\calF)$.
We are mainly interested in the case $\calF=AL^1(\Omega)$, but also $\calF=HL^1(\Omega)$,
$\calF=SL^1(\Omega)$ will play a role. In the last case, \eqref{qi} must be replaced by the inequality
\begin{equation*}\label{qin}\mu(f)\le \int_\Omega f\, dA,\qquad f\in SL^1(\Omega).\end{equation*}

It is easy to see that $Q(\mu,SL^1)\subset Q(\mu,HL^1)\subset Q(\mu,AL^1)$ and that a solid domain belongs simultaneously to the classes
$Q(\mu,HL^1)$ and $Q(\mu,AL^1)$.

\smallskip

The above classes are conveniently interpreted in terms of the logarithmic potentials
\begin{equation*}\label{norm}U^\mu:=\ell*\mu,\qquad U^\Omega:=U^{\1_\Omega\, dA},\qquad \text{where}\quad \ell(z):=\frac 1 {2}\log\frac 1 {|z|}.\end{equation*}

For example, we have that $\Omega\in Q(\mu,AL^1)$ if and only if $\d U^\mu=\d U^\Omega$ on $\Omega^e$
and $\Omega\in Q(\mu,SL^1)$ if and only if $U^\mu=U^\Omega$ on $\Omega^e$ and $U^\mu\ge U^\Omega$ on $\C$.

\smallskip

Given these proviso, we can formulate our basic problem in a succinct way (cf. \cite{G90}).

\smallskip

(Q). \textit{Determine whether or not there exists a functional $\mu$ such that
the class $Q(\mu,AL^1)$ contains two distinct, solid domains.}

\begin{theorem} \label{settles} The following uniqueness results are known.
\begin{enumerate}[label=(\roman*)]
\item \label{no38} If $\Omega_1,\Omega_2\in Q(\mu,AL^1)$ are star-shaped with respect to
a common point, then $\Omega_1=\Omega_2$.
\item \label{sa82} If there exists a solid domain $\Omega\in Q(\mu,SL^1)$, then this $\Omega$ is the unique solid quadrature domain, even within
the class $Q(\mu,AL^1)$.
\item \label{sa99} If $\mu$ is a positive measure of total mass $m$ and if $\supp\mu$ is contained in a disc of radius $r$ where $r^2< m$, then each solid domain $\Omega$ of class $Q(\mu,AL^1)$ is obtainable from $\mu$ by partial balayage, and so it belongs to $Q(\mu,SL^1)$.
\end{enumerate}
\end{theorem}

\begin{proof}[Remark on the proof] Part \ref{no38} is due to Novikov \cite{Novikoff1938}, cf. also \cite{Isakov,Gardiner2008}.
Part \ref{sa82} was proved in Sakai's book \cite{Sakai1982}, using the technique of partial balayage. An alternative proof is found in the paper \cite{G90} by Gustafsson.
The statement \ref{sa99} was likewise proved by Sakai using partial balayage, see the papers \cite{Sakai1998,Sakai1999,G07}.
\end{proof}



The ``classical'' setting corresponds to \textit{point-functionals}, i.e., functionals $\mu$ of the form
\begin{equation*}\label{pf}\mu(f)=\sum_{i=1}^m\sum_{j=0}^{n_i-1} c_{ij}f^{(j)}(a_i),\qquad f\in AL^1(\Omega)\end{equation*}
where $a_i$ are some points in $\Omega$ and $c_{ij}$ some complex numbers. A quadrature domain of this type is said to be of \textit{order} $n_1+\cdots+n_m$. When $\mu$ contains no derivatives, i.e., when $\mu(f)=\sum_{i=1}^n c_if(a_i)$ we speak of a \textit{pure} point-functional. (Cf. \cite{Davis74}.)

\begin{example} Theorem \ref{settles}\ref{sa82} completely settles the uniqueness problem for subharmonic quadrature domains. The following example due to Gustafsson  shows that the question (Q)
for analytic test functions is of a different kind.

It is shown in \cite[Section 4]{Gu88} that there exists a quadrature domain $\Omega$ having the appearance in Figure \ref{Figure1}, satisfying a three-point
identity $\int_\Omega f\, dA=c_1f(a_1)+c_2f(a_2)+c_3f(a_3)$ where $c_1,c_2,c_3>0$.
\begin{figure}
\includegraphics[width=.4\textwidth]
{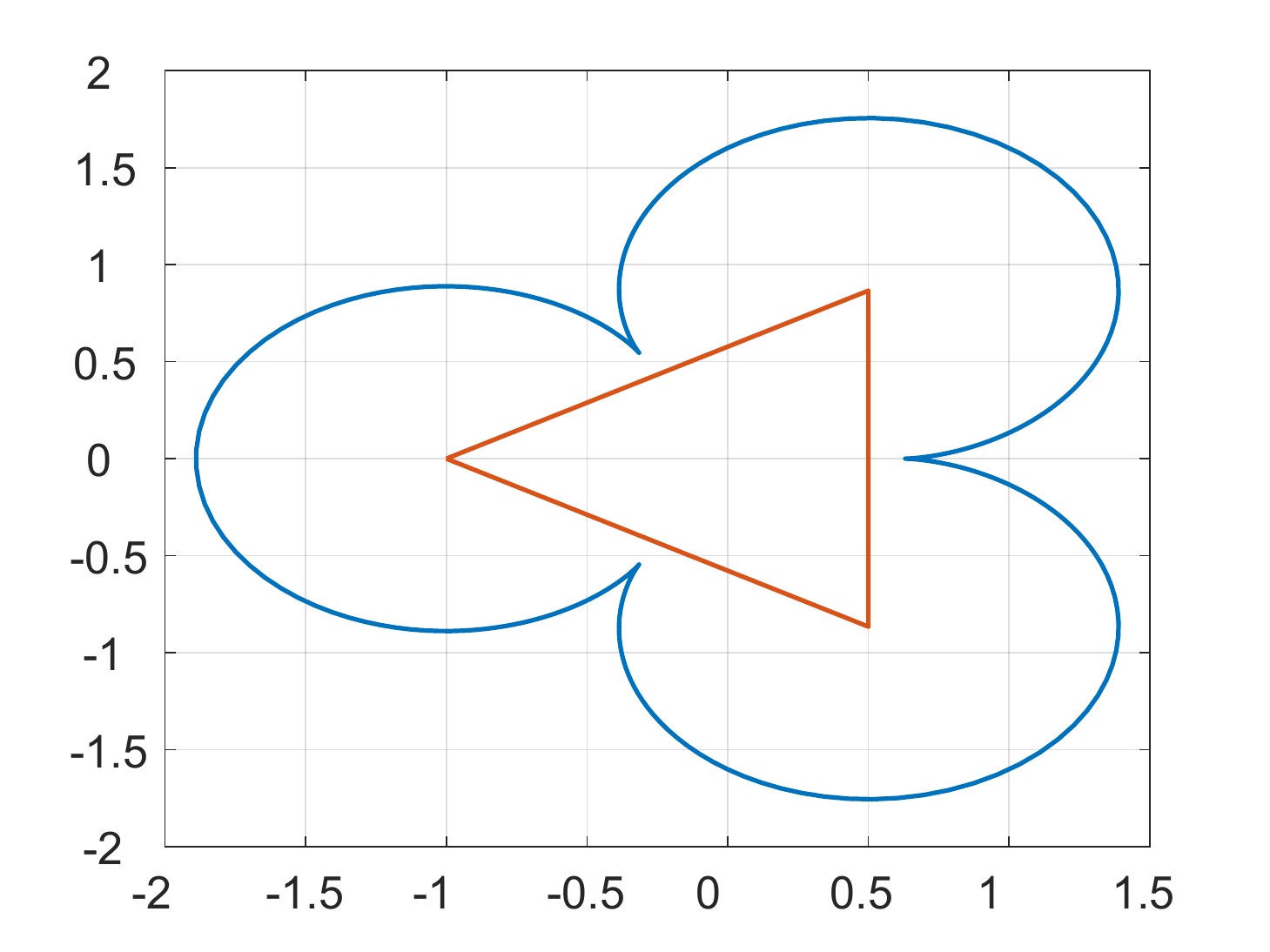}
\caption{Gustafsson's example of a domain in $Q(\mu,HL^1)\setminus Q(\mu,SL^1)$. The triangle depicts the convex hull of the nodes $a_1,a_2,a_3$.}
\label{Figure1}
\end{figure}
It is known (see e.g. \cite[Theorem 2.1]{G07}) that if $\Omega\in Q(\mu,SL^1)$ and if $\d\Omega$ has cusps, then those cusps must be contained in the convex hull of the
nodes $a_i$. Hence the quadrature domain in Figure \ref{Figure1} is not subharmonic.
\end{example}

Now fix a solid quadrature domain $\Omega\in Q(\mu,AL^1)$ containing the origin $0$. Let $\fii:\D\to\Omega$ be the conformal map normalized by $\fii(0)=0$ and $\fii'(0)>0$. Recall that
$\Omega$ is uniquely determined by $\fii$ via Riemann's mapping theorem.

The following theorem, which gives a nontrivial relation for $\fii$, will be the main tool in our subsequent investigations.

\begin{theorem}\label{theorem: structure of univalent function} Suppose that $\mu$ has compact support in $\Omega$ and let $\nu$ be the pullback of $\mu$ to
$\D$, i.e., $\nu(g)=\mu(g\circ\fii^{-1})$ for $g\in AL^1(\D)$.
Then $\phi$ satisfies the relation
	\begin{equation}\label{nonlin}
	\phi(z)=
\nu_\lambda^*\left(\frac{z}{\overline{\phi'(\lambda)}(1-z\bar{\lambda})}\right),\qquad (z\in\mathbf{D})
	\end{equation}
	where $\nu_\lambda^*(g):=\overline{\nu_\lambda(\bar{g})}$ acts on integrable anti-analytic functions $g(\lambda)$. 

Conversely, if $\fii$ is any univalent solution to \eqref{nonlin}, then the domain $\Omega=\fii(\D)$ is of class $Q(\mu,AL^1)$.
\end{theorem}

This result is not very easy to spot in the literature, but it has in fact been noticed earlier in somewhat different guises.
The first proof might be due to Davis, see \cite[Chapter 14]{Davis74}, cf. also \cite[Section 5]{D72}.
Since the result will be central for what follows, we include an alternative proof (that we have found independently) in Section \ref{sec:master}.

In the special case of a pure point-functional $\mu(f)=\sum_1^n c_i f(a_i)$, the relation \eqref{nonlin} takes the form
\begin{equation}\label{meq}\fii(z)=\sum_{i=1}^n \frac {\bar{c}_i}{\bar{w}_i}\frac {z}{1-\bar{\lambda}_i z},\qquad (w_i=\fii'(\lambda_i),\, \fii(\lambda_i)=a_i),\end{equation}
which appears implicitly
in
e.g.~the books \cite{Shapiro1992,Gustafsson2006,Varchenko1992}.

The main idea behind our approach is to ``solve'' functional relations such as \eqref{meq} by using techniques from algebraic geometry. To set up a suitable system of polynomial equations we differentiate \eqref{meq} and substitute $z=\lambda_j$, giving
\begin{equation}\label{meq2}
w_j=\sum_{i=1}^n\frac {\bar{c}_i}{\bar{w}_i}\frac 1 {(1-\lambda_j\bar{\lambda}_i)^2},\qquad
 a_j=\sum_{i=1}^n\frac {\bar{c}_i}{\bar{w}_i}\frac {\lambda_j} {1-\lambda_j\bar{\lambda}_i},\qquad j=1,\ldots,n,\end{equation}
where the unknown complex numbers $\lambda_i$ and $w_i$ are subject to the constraints
\begin{equation}\label{con2}\lambda_1=0,\quad |\lambda_2|<1,\quad \cdots\quad |\lambda_n|<1,\quad w_1>0.\end{equation}

The appearance of inequalities and complex-conjugates means that we are considering the \textit{real} semi-algebraic geometry of a particular system of rational functions. Such semi-algebraic systems, i.e.~those having the special structure of \eqref{meq2}, \eqref{con2} have, to the best of our knowledge, not been systematically studied before.

It is of course possible that no univalent solution to \eqref{meq2}, \eqref{con2} exists; for example the quadrature identity $\int_\Omega f\, dA=f(a_1)+f(a_2)$ implies that $\Omega$ is the disjoint union $D(a_1,1)\cup D(a_2,1)$ if $|a_1-a_2|\ge 2$. However,
after having studied exact solutions for many examples of lower order quadrature domains, we find ``empirically'' the pattern that 
the system tends to have at most
one
solution which may or may not give rise to a univalent mapping $\fii$. The non-univalent solutions 
fail to be
\textit{locally} univalent, i.e., they (still, empirically) satisfy $\fii'=0$ somewhere in the disc $\D$. For a different type of quadrature domain, not described by \eqref{meq2} and \eqref{con2}, a similar observation was made by Ullemar in \cite{Ullemar1980}.

\begin{remark} The non-univalent solutions $\fii$ are believed to represent quadrature domains on Riemann surfaces with branch points. Such q.d.'s
are studied in the references \cite{GLR,SA88,Skinner}.
\end{remark}

\begin{remark}
Note that our method relies on knowledge
of \textit{all} solutions to \eqref{meq2}, \eqref{con2}. To find one or several approximate solutions, one can of course try to apply numerical methods, such as Newton's iterative
method. By appropriately choosing different initial data, we may indeed obtain
solutions by such methods in a relatively short time for $n$ up to $10$.
However, since the number of solutions to the system is unknown, it is impossible to know when one has found all solutions,
so this kind of information is of no use when studying the uniqueness question for quadrature domains. Another problem with a numerical approach is that systems such as \eqref{meq2}, \eqref{con2} tend to be quite sensitive to small perturbations of the quadrature data $\{c_i,a_i\}_1^n$.
\end{remark}

\begin{remark} In connection with uniqueness problems in the gravi-equivalent sense, it is pertinent to point (besides the sources already mentioned) to the early works of Zidarov and Zhelev in the context of geophysics,
see e.g.~\cite{zidarov1970obtaining,zidarov1973method}. (We thank one of the referees for this remark.)
\end{remark}

The literature on quadrature domains is vast, and we have at this stage omitted to mention several important aspects. A somewhat fuller picture is given
in Section \ref{concrem}, where we briefly compare a few of the more well-known techniques that have been developed
over the years, such as Laplacian growth and Schottky-Klein functions. 

\smallskip

To illustrate the challenges involved in studying the uniqueness of quadrature domains, we now give an example demonstrating the subtlety of the problem
even for a q.d of order 2.

\begin{example}
Let $\Omega_1$ be the solid quadrature domain obtained from a monopole with charge 1/2 and a dipole with strength $\sqrt{3}/18$ placed at the origin, i.e.
\begin{equation*}
	\int_{\Omega_1}f\ dA=\frac{1}{2}f(0)+\frac{\sqrt{3}}{18}f'(0),\qquad f\in AL^1(\Omega_1).
\end{equation*}
For quadrature domains of this type, Aharonov and Shapiro have proved uniqueness in \cite{Aharonov1976}; in fact
$\Omega_1$ is determined as the image of $\D$ under the conformal map $p(z)=\frac{\sqrt{3}}{6}(2z+z^2)$. The boundary $\d\Omega_1$ is a cardioid with
a cusp at $p(-1)=\frac{-1}{\sqrt{12}}$, see Fig. \ref{fig: il ex}.

Let us now construct a similar q.d. $\tilde{\Omega}_2$ but only using point charges,
\begin{equation*}
	\int_{\tilde{\Omega}_2}f\ dA=-\frac{1}{2}f(0)+f(a_2),\qquad f\in AL^1(\tilde{\Omega}_2).
\end{equation*}
Clearly both $\Omega_1$ and $\tilde{\Omega}_2$ have area $1/2$. We shall choose the parameter $a_2$ real, such that the boundary of $\tilde{\Omega}_2$ has a cusp.

Since the parameters $c_i,w_i,\lambda_i$ in \eqref{meq} must be real in this case, the conformal map $\tilde{\phi}:\D\to\tilde{\Omega}_2$ takes the form
\begin{equation}\label{meq7}
	\tilde{\phi}:z\mapsto\frac{c_1z}{w_1}+\frac{c_2}{w_2}\frac{z}{1-\lambda_2z}
\end{equation}
where $c_1=-1/2$ and $c_2=1$, $\lambda_2$ the pre-image of $a_2$, and $w_1=\tilde{\phi}'(0),\ w_2=\tilde{\phi}'(\lambda_1)$.

Writing $\nu=-9+18\,i\sqrt {2}$, we find
that there is a unique choice of $a_2$ producing a cusp, namely
\begin{align*}
	a_2=&\frac 1 {2\,\nu^{1/3}}\,\sqrt{-\nu^{1/3} \left( 3\,i\sqrt {3}\,\nu^{2/3}-27\,i\sqrt {3}+3\, \nu^{2/3}+4\,\nu^{1/3}+27\right)}
\approx 0.1316.
\end{align*}
The remaining parameters in the mapping function may be obtained by solving the system corresponding to \eqref{meq7} using the method of ``RCTD'' described in Section \ref{sec:background}. The result is
\begin{align*}
	\lambda_2&=\frac{3a_2}{\sqrt{3a_2^2+3}}\nonumber
	\approx 0.2261,\nonumber\\
	w_{{2}}&={\frac {-c_{{2}}{\lambda_{{2}}}^{3} \left( -{\lambda_{{2}}}^{2
			}c_{{1}}+{\lambda_{{2}}}^{2}c_{{2}}+{a_{{2}}}^{2} \right) }{a_{{2}}
			\left( -c_{{1}}{\lambda_{{2}}}^{6}+c_{{2}}{\lambda_{{2}}}^{6}+{
				\lambda_{{2}}}^{4}{a_{{2}}}^{2}+2\,c_{{1}}{\lambda_{{2}}}^{4}-2\,{
				\lambda_{{2}}}^{2}{a_{{2}}}^{2}-{\lambda_{{2}}}^{2}c_{{1}}-{\lambda_{{2}}}^{2}c_{{2}}+{a_{{2}}}^{2} \right) }}\nonumber\approx 0.6674,\nonumber\\
	w_1&={\frac {c_{{1}}\lambda_{{2}} \left( {\lambda_{{2}}}^{2}-1 \right) w_{{
					2}}}{w_{{2}}{\lambda_{{2}}}^{2}a_{{2}}-w_{{2}}a_{{2}}+\lambda_{{2}}c_{
				{2}}}}\nonumber\approx 0.5016.
\end{align*}
This gives a cusp at
$\tilde{\phi}(-1)=-\frac{c_1}{w_1}-\frac{c_2}{w_2}\frac{1}{1+\lambda_2}\nonumber \approx -0.2253$.

\begin{figure}[!t]
	\centering
	\includegraphics[width=0.6\textwidth]{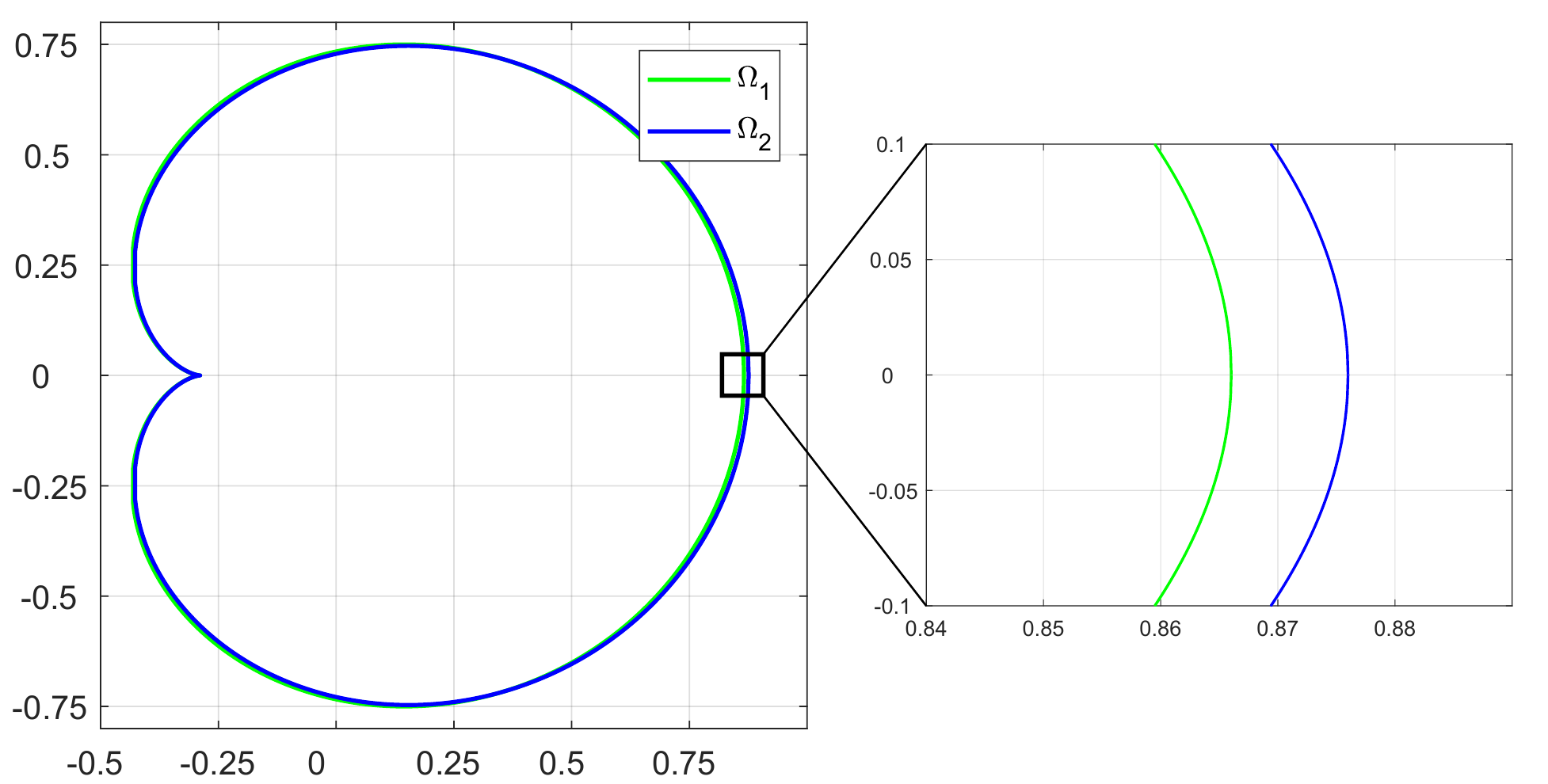}
	\caption{The domains $\Omega_1$ and $\Omega_2$.}
	\label{fig: il ex}
\end{figure}

A translate of $\tilde{\Omega}_2$ by $\alpha:=p(-1)-\tilde{\phi}(-1)\approx -0.06333$ leads to a domain $\Omega_2=\tilde{\Omega}_2+\alpha$ having a cusp at the point $p(-1)$ and satisfying the quadrature identity
\begin{equation*}
	\int_{\Omega_2}f\ dA=-\frac{1}{2}f(\alpha)+f(a_2+\alpha),\qquad f\in AL^1(\Omega_2).
\end{equation*}

The resemblance between $\Omega_1$ and $\Omega_2$ (Fig. \ref{fig: il ex}) is striking, even though they admit completely different quadrature identities.
 The similarity between $\Omega_1$ and $\Omega_2$ indicates that their potentials should be similar, and in terms of numerical values they are. But there is one essential difference between the two: the potential of $\Omega_1$ is exactly determined by two terms in its multipole expansion while the potential for $\Omega_2$ needs the entire infinite series. In detail we have
\begin{align*}
	 U^{\Omega_1}(z)&=\frac{1/2}{2\pi}\log\frac{1}{|z|}+\frac{\sqrt{3}/18}{2\pi}\mathrm{Re}\left(\frac{1}{z}\right)\\
	 U^{\Omega_2}(z)&\approx\frac{1/2}{2\pi}\log\frac{1}{|z|}+\frac{0.09998}{2\pi}\mathrm{Re}\left(\frac{1}{z}\right)+\ldots
\end{align*}
and for comparison note $\sqrt{3}/18\approx0.09623$.
From this example we see that two very similar domains may have fundamentally different potentials.
\end{example}

\section{The master formula}\label{sec:master}

As previously stated, Theorem \ref{theorem: structure of univalent function} appears (in equivalent form) in \cite[Eq. (14.12)]{Davis74}. We shall here give a different derivation.

\smallskip

Consider the univalent map $\fii:\D\to\Omega$ normalized by $\fii(0)=0$ and $\fii'(0)>0$ where $\Omega\in Q(\mu,AL^1)$ and where the distribution
$\mu$ is assumed to be of compact support in $\Omega$. Our point of departure is
Poisson's equation
\begin{equation*}\label{eq: Poisson}\Delta U^\mu=-\mu.\end{equation*}

Taking the distributional $\d$-derivative of $-4U^\mu$ we obtain the \emph{Cauchy transform}
\begin{equation*}\label{ct}
\mathcal{C}\mu(z):=\mu*k(z)=\mu(k_z),
\end{equation*}
where $k(\lambda):=\lambda^{-1}$ denotes the \textit{Cauchy kernel} and $k_z(\lambda):=k(z-\lambda)$.
Since $-4\d U^\mu=\calC\mu$, \eqref{eq: Poisson}
says that
$\dbar\calC\mu=\mu$. In particular $\calC\mu$ is holomorphic on
$\mathbb{C}\setminus\supp \mu$.
Taking $f=k_z$ $(z\not\in \Omega)$ in the quadrature identity \eqref{qi} we see that
\begin{equation}\label{qua}\d U^\Omega=\d U^\mu\qquad \text{on}\qquad \C\setminus\Omega.\end{equation}
In fact, an application of Bers' approximation theorem from \cite{B65} shows that \eqref{qua} is equivalent to \eqref{qi}.
Now consider the ``Schwarz potential'' $u$ defined by
\begin{equation*}
u=\mathcal{C}(\1_\Omega-\mu)=-4\partial(U^\Omega-U^\mu),
\end{equation*}
which is zero for $z\in\Omega^e$. Using the continuity of $\partial U^\Omega$
we get $u=0$ also on $\partial\Omega$. Moreover, Poisson's equation gives
$\dbar u=\1_\Omega$ on $\CC\setminus\supp \mu$.
Hence the function
$$S(z):=\bar{z}-u(z)$$
is holomorphic on $\Omega\setminus\mathrm{supp}\ \mu$ and continuous up to the boundary $\partial\Omega$, while satisfying $S(z)=\bar{z}$ for $z\in\d\Omega$.
This determines $S$ as
 the \emph{Schwarz function} for the boundary curve $\d \Omega$, cf. \cite{Davis74,Shapiro1992}.

The following lemma is well-known, see e.g. \cite{Davis74,Shapiro1992}.

\begin{lemma}\label{lemma: holomorphic continuation}
The conformal mapping $\phi:\mathbf{D}\to\Omega$ extends holomorphically across $\mathbf{T}$ to an analytic function on the disk $D(0,R)$ for some $R>1$.
\end{lemma}

\begin{proof}
As we saw above, the function $S\circ\phi$ is defined and holomorphic in some annulus $1-\epsilon<|z|<1$, continuous up to the boundary and satisfies $S(\phi(z))=\overline{\phi(z)}$ when $z\in\mathbf{T}$.
Likewise, the function $\phi^*(z)=\overline{\phi(\overline{z})}$ is holomorphic in $\mathbf{D}$ and continuous up to the boundary, and we have the relation
	\begin{equation*}
	S(\phi(z))=\phi^*(1/z),\qquad z\in\mathbf{T}.
	\end{equation*}
	Now, $\phi^*(1/z)$ is holomorphic in the exterior of $\mathbf{D}$, so the above formula shows that the functions $S\circ\phi(z)$ and $\phi^*(1/z)$ are analytic continuations of each other across the circle $\mathbf{T}$. In particular, $\phi^*(1/z)$ is analytically continuable inwards across $\mathbf{T}$, which means that $\phi^*$ as well as $\phi$ are analytically continuable outwards across $\mathbf{T}$, to some disc $D(0,R)$ with $R>1$.
\end{proof}

\begin{lemma}\label{Lemma: CauchyTransform} We have that
$
	\mathcal{C}\left[\overline{\phi'}\cdot\1_\mathbf{D}\right](z)=\overline{\phi(z)}
,\, z\in\mathbf{T}.$
\end{lemma}

\begin{remark} For an absolutely continuous measure $\nu=f\, dA$, we prefer to denote its Cauchy transform by $\Cauchy f$ rather than $\Cauchy \nu$.
\end{remark}

\begin{proof}[Proof of Lemma \ref{Lemma: CauchyTransform}]
	Fix a point $z=e^{i\theta}\in\mathbf{T}$ and a positive number $r<1$ and put
	\begin{align*}
	I_r=\,&\mathcal{C}\left[\overline{\phi'}\cdot \1_\mathbf{D}\right](rz)=\int_\mathbf{D}\frac{\overline{\phi'(\lambda)}}{re^{i\theta}-\lambda}\  dA(\lambda) = re^{-i\theta}\int_{r\mathbf{D}}\frac{\overline{\phi'(r\zeta e^{i\theta})}}{1-\zeta}\ dA(\zeta).
	\end{align*}

Set $
	\phi(\lambda)=\sum_{j=0}^\infty c_j\lambda^j$, where
$\sum_{j=0}^\infty |c_j|<\infty$ by Lemma \ref{lemma: holomorphic continuation}.
Inserting the expansion $\frac 1 {1-\zeta}=\sum\zeta^j$ we find that
	\begin{align*}
	I_r
=\sum_{j=1}^\infty\bar{c}_j\left(re^{-i\theta}\right)^jr^{2j}.
	\end{align*}
Since $\sum |c_j|<\infty$ we may pass to the limit as $r\nearrow 1$, leading to
	\begin{align*}
	\lim_{r\to 1}I_r&=\lim_{r\to 1}\sum_{j=1}^\infty\bar{c}_j\left(re^{-i\theta}\right)^jr^{2j}=\sum_{j=1}^\infty\bar{c}_j\left(e^{-i\theta}\right)^j=\overline{\phi(z)}.
	\end{align*}
The proof of the lemma is complete.
\end{proof}

\begin{proof}[Proof of Theorem \ref{theorem: structure of univalent function}]
The quadrature identity \eqref{qi} pulls back to
	\begin{equation}\label{laeq}
	\int_\Omega f\ dA=\int_\mathbf{D}(f\circ\phi)\cdot|\phi'|^2\ dA=\nu(f\circ\phi)
	\end{equation}
where $\nu(g)=\mu(g\circ\phi^{-1})$.

Given an arbitrary $f\in AL^1(\Omega)$ we define a function $g\in AL^1(\D)$ by
$$g=(f\circ\phi)\cdot\phi'.$$ The identity \eqref{laeq} can be written as
	\begin{equation}\label{eq: integral}
	\int_\mathbf{D}g\overline{\phi'}\ dA=\nu\left(\frac{g}{\phi'}\right).
	\end{equation}
Now fix a point $z\in\T$ and choose $g$ to be the Cauchy-kernel $g=k_z$.
With this choice, \eqref{eq: integral} takes the form
	\begin{equation*}
	\mathcal{C}\left[\overline{\phi'}\cdot\1_\mathbf{D}\right](z)=\Cauchy\left[ \frac 1 {\fii'}\cdot \nu\right](z)
,\qquad z\in\T.
	\end{equation*}
	By Lemma \ref{Lemma: CauchyTransform} this is equivalent to
	\begin{equation*}
	\overline{\phi(z)}=\mathcal{C}\left[\frac{1}{\phi'}\cdot \nu\right](z)=\nu_\lambda\left(\frac{1}{\phi'(\lambda)(1/\bar{z}-\lambda)}\right),\qquad z\in\T.
	\end{equation*}
Taking complex-conjugates and considering the analytic continuation to $\D$ we obtain
\begin{equation}\label{end}\fii(z)=\nu_\lambda^*\left[\frac{z}{\overline{\phi'(\lambda)}(1-z\bar{\lambda})}\right],\qquad (z\in\D),\end{equation}
as desired.

Conversely, if $\fii$ is univalent (and normalized) in $\D$ and satisfies \eqref{end}, we may read backwards and deduce \eqref{laeq}, so $\Omega=\fii(\D)$
belongs to $Q(\mu,AL^1)$ where $\mu$ is the push-forward of $\nu$.
\end{proof}

We conclude this section with three examples of applications of Theorem \ref{theorem: structure of univalent function}, which are known from the literature on quadrature domains.

\begin{example} Let us compare Theorem \ref{theorem: structure of univalent function} with the computations in Shapiro's book \cite[Proposition 3.2]{Shapiro1992}.
For this purpose we fix a pure point functional
$\mu=\sum_{i=1}^n c_i\delta_{a_i}$, which pulls back to
\begin{equation*}
\nu(g)=\mu(g\circ\varphi^{-1})=\sum_{i=1}^n c_ig(\varphi^{-1}(a_i))=\sum_{i=1}^n c_ig(\lambda_i),\qquad (\lambda_i=\fii^{-1}(a_i)).
\end{equation*}
By Theorem \ref{theorem: structure of univalent function} we know that an arbitrary solid domain $\Omega\in Q(\mu,AL^1)$ is of the form
$\Omega=\fii(\D)$ where $\fii$ is univalent and normalized and satisfies
\begin{align}\label{bsys}
\varphi(z)
=\sum_{i=1}^n\frac{\overline{c}_i}{\overline{\varphi'(\lambda_i)}}\frac{z}{1-z\overline{\lambda}_i}
\end{align}
\end{example}

Given such a $\fii$, we put $S(\phi(z)):=\phi^*(1/z)$ and note that $S$ is the Schwarz function for $\d\Omega$.
In view of \eqref{bsys}, $S$ is a meromorphic function with
simple poles at $z=a_j$. As $z\to\lambda_j$ the dominant term in $S\circ\phi$ satisfies
\begin{equation*}
S(\phi(z))\sim\frac{c_j}{\phi'(\lambda_j)}\frac{1}{z-\lambda_j}.
\end{equation*}
From this we get that (as $z\to a_j$)
\begin{equation*}
S(z)\sim\frac{c_j}{\phi'(\lambda_j)}\frac{1}{\phi^{-1}(z)-\lambda_j}=\frac{c_j}{\phi'(\lambda_j)}\frac{z-a_j}{\phi^{-1}(z)-\phi^{-1}(a_j)}\frac{1}{z-a_j}\sim \frac{c_j}{z-a_j}.
\end{equation*}
The residues of $S$ are thus just Res$(S;a_j)=c_j$ for all $j$.
Since $S(z)=\bar{z}$ on $\d\Omega$, an application of Green's theorem and the Residue theorem now gives
\begin{align*}
\int_\Omega f\ dA&=\frac{1}{2\pi i}\int_\Gamma f(z)\overline{z}\ dz=\frac{1}{2\pi i}\int_\Gamma f(z)S(z)\ dz = \sum c_jf(a_j)
\end{align*}
where $\Gamma$ is the positively oriented boundary of $\Omega$. We have shown again that $\Omega\in Q(\mu,AL^1)$.

We remark that a similar proof applied to a more general point functional $\mu(f)=\sum c_{ij}f^{(j)}(a_i)$ gives the well known result (see \cite[Theorem 3.3.1]{Gustafsson2014}) that a solid domain $\Omega$ is a quadrature domain of finite order if and only if
each conformal map $\fii:\D\to\Omega$ is a rational function, if and only if the Schwarz function of $\d\Omega$ extends to a meromorphic function in $\Omega$.

\begin{example} \label{Example: monopole and dipole}
Let $\mu$ be a linear combination of a monopole and a dipole at the origin, i.e. $\mu(f)=M_0f(0)+M_1f'(0)$. The action of the pullback is then given by
\begin{align*}
\nu(g)
=\mu(g\circ\varphi^{-1})
=M_0g(0)+\frac{M_1}{\varphi'(0)}g'(0),\qquad g\in AL^1(\D).
\end{align*}
Applying Theorem \ref{theorem: structure of univalent function}, we find that a normalized conformal map $\fii:\D\to\Omega$, where $\Omega\in Q(\mu,AL^1)$, must satisfy
\begin{equation*}\label{eq: polynomial map deg two}
\overline{\varphi(z)}=\left(\frac{1}{\varphi'(0)}M_0-\frac{\varphi''(0)}{\varphi'(0)^3}M_1\right)\bar{z}\nonumber+\frac{M_1}{\varphi'(0)^2}\bar{z}^2.
\end{equation*}
Hence $\fii(z)$ is a polynomial of degree two. To determine this polynomial, we need to determine the derivatives $\fii'(0)$, $\fii''(0)$. The computation is
 postponed to Subsection \ref{section:AS}, after we have discussed some algebraic prerequisites.
\end{example}

\begin{example}\label{ulc} Following Davis \cite[pp. 162-166]{Davis74} we now take
$\mu$ be a line charge with linear density $h$ on the segment $[-a,a]$ of the $x$-axis, i.e., $$\displaystyle \mu(f)=\int_{-a}^af(x)h(x)\ dx.$$
Suppose that $\Omega=\fii(\D)\in Q(\mu,AL^1)$;
the pullback $\nu$ by $\fii$ is then given by
\begin{equation*}
\nu(g)=\mu(g\circ\varphi^{-1})=\int_{-a}^{a}g(\varphi^{-1}(x))h(x)\ dx=\int_{-\lambda}^\lambda g(w)h(\fii(w))\fii'(w)\, dw,\qquad (\lambda=\varphi^{-1}(a)).
\end{equation*}
Applying Theorem \ref{theorem: structure of univalent function} we see that $\fii$ must satisfy the functional equation
(cf. \cite[Eq. (14.25)]{Davis74})
\begin{equation*}
\overline{\varphi(z)}=\bar{z}\int_{-\lambda}^\lambda \frac {h(\fii(w))}{1-\bar{z}w}\, dw.\end{equation*}
In particular, if we specialize to a uniform line charge $h\equiv 1$, we obtain
\begin{equation*}\fii(z)=\log\left(\frac{1+z\lambda}{1-z\lambda}\right).
\end{equation*}
This relation was found by Davis (see Eq. (14.13)), where it is also shown that if $a=1$ then $\lambda=\sqrt{\tanh(1/2)}=.6798\cdots<1$. It is then easy to
see that the above map $\fii$ is well-defined by choosing the standard branch of the logarithm. We have shown that there exists a unique solid domain $\Omega$ in the class $Q(\1_{[-1,1]}(x)\, dx,AL^1)$; a picture is given in Figure \ref{Figure3}.
\begin{figure}
\includegraphics[width=.4\textwidth]
{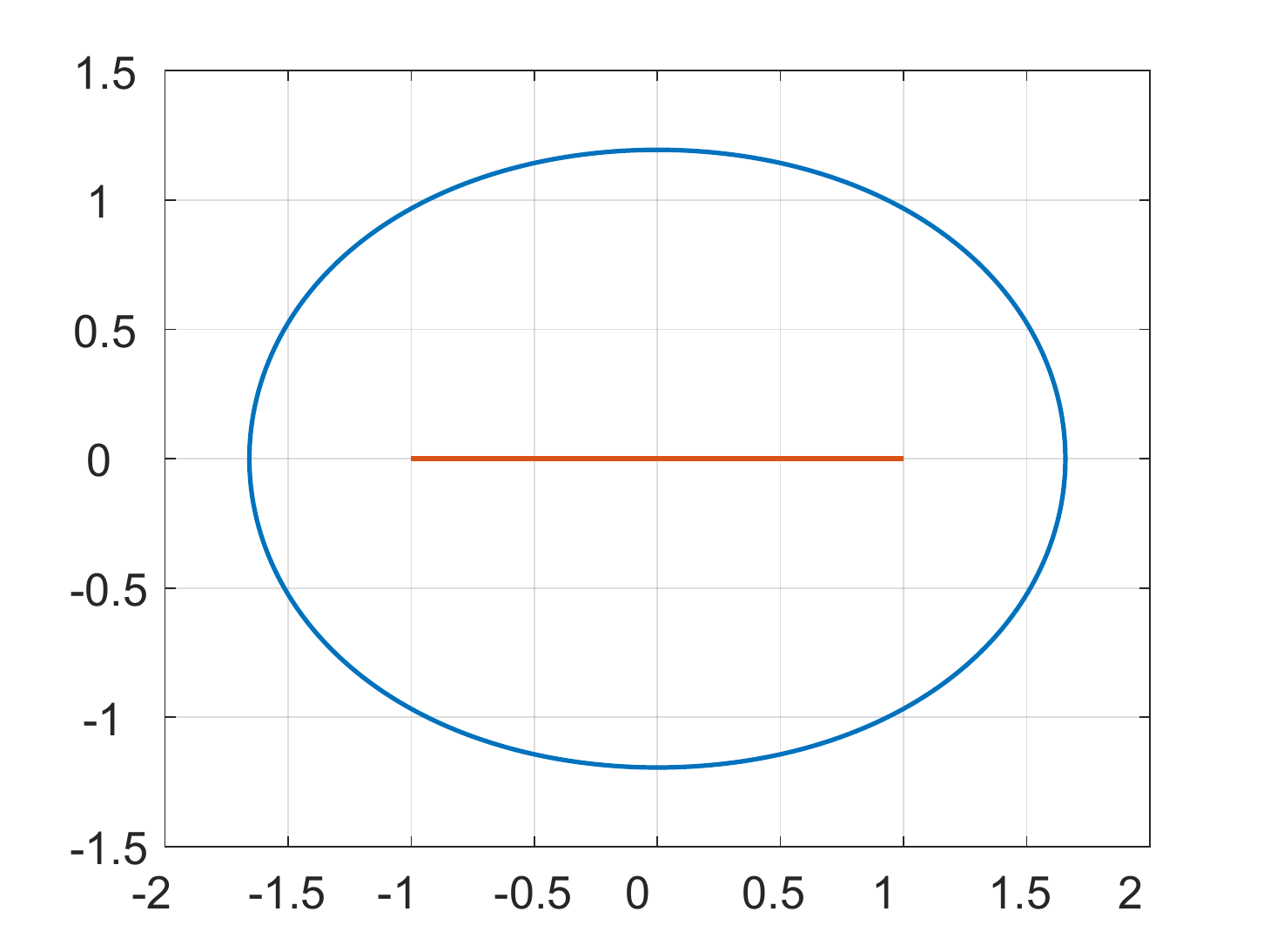}
\caption{The q.d. $\Omega$ generated by the linear density $d\mu(x)=\1_{[-1,1]}(x)\, dx$.}
\label{Figure3}
\end{figure}
\end{example}

\section{Schur-Cohn's test}\label{Section:SchurCohn}

As stated earlier, the mapping problem \eqref{meq} may have non-univalent solutions $\fii$.
The Schur-Cohn test provides a convenient way of discarding such
$\fii$ that fail to be \textit{locally} univalent.
One might hope that this procedure would leave us with at most one univalent $\fii$, thus settling the
uniqueness problem. As we will see later, this is indeed the case for a large class of examples.

The Schur-Cohn test is well known and quite elementary, see \cite{Henrici74}.
For reasons of completeness, we have found it convenient to briefly review the main ideas behind it here.
In what follows, we let $\mathcal{P}_n$ denote the space of all polynomials $p$ of degree at most $n$,
\begin{equation*}\label{eq: polynomial}
p(z)=a_0+a_1z+\ldots+a_nz^n,\qquad a_0,\ldots,a_n\in\mathbb{C}.
\end{equation*}
\begin{lemma}\label{lemma: SC} If $n\ge 1$ then

(i) if $|a_0|>|a_1|+\ldots+|a_n|$ then $ p(z)\neq 0$ for all $z$ with $|z|\le 1$,

(ii) if $ p(z)\neq 0\ \mathrm{for\ all}\ z\ \mathrm{with}\ |z|\le 1$ then $|a_0|>|a_n|$.
\end{lemma}

\begin{proof}
	(i) If $|a_0|>|a_1|+\ldots+|a_n|$ then $|a_1z+\ldots+a_nz^n|\le|a_1|+\ldots+|a_n|<|a_0|$ for
 $|z|\le 1$,
	so $|p(z)|\ge |a_0|-|a_1z+\ldots+a_nz^n|>0$.
	
	(ii) Assume $p(z)\neq 0$ for all $z$ with $|z|\le 1$. Then $p(0)=a_0\neq 0$ which leads to two cases: (1)
if $a_n=0$ then obviously $|a_0|>|a_n|$; (2) if $a_n\neq 0$ we have $p(z)=a_n(z-z_1)\cdot\ldots\cdot(z-z_n)$ where $z_n$ are the zeros of $p$. But then $p(0)=a_0=a_n(-z_1)\cdot\ldots\cdot(-z_n)$ and thus $|a_0|>|a_n|$ since $|z_k|>1$ for $k=1,\ldots,n$.
\end{proof}

For each $p\in\mathcal{P}_n$ the \textit{reciprocal polynomial} $p^\#\in\mathcal{P}_n$ is defined by
\begin{equation*}
p^\#(z)=z^n\cdot\overline{p(1/\overline{z})}=\overline{a}_n+\overline{a}_{n-1}z+\ldots+\overline{a}_1z^{n-1}+\overline{a}_0z^n.
\end{equation*}

Let us now define the \emph{Schur transform}, $S_n:\ \mathcal{P}_n\to\mathcal{P}_{n-1}$ by
\begin{equation*}
(S_np)(z)=\overline{a}_0p(z)-a_np^\#(z)=\sum_{k=0}^{n-1}(\overline{a}_0a_k-a_n\overline{a}_{n-k})z^k,\qquad (p\in\mathcal{P}_n).
\end{equation*}

We note a few simple facts pertaining to these objects.

First, $|p^\#|=|p|$ on $\T$. Moreover, every zero of $p$ on $\T$ is also a zero of $p^\#$ and is thus a zero of $S_np$. Finally,
\begin{equation*}
(S_np)(0)=|a_0|^2-|a_n|^2\in\mathbb{R}.
\end{equation*}

We now construct a chain of polynomials $p_0,p_1,\ldots,p_n$ starting with $p_0=p$ and then taking successive Schur transforms,
\begin{equation*}
p_1=S_np_0,\quad p_2=S_{n-1}p_1,\qquad \cdots\qquad p_n=S_1p_{n-1}.
\end{equation*}
The last polynomial $p_n$ is an element in $\mathcal{P}_0$ and thus is a constant.

\begin{lemma}\label{SchurLemma} If $p_k$ has no zeros on the unit circle $\mathbf{T}$ and $p_{k+1}(0)>0$, then $p_k$ and $p_{k+1}$ have equally many zeros in $\overline{\mathbf{D}}$.
\end{lemma}

\begin{proof}
Let $p_k(z)=b_0+b_1z+\ldots+b_{n-k}z^{n-k}$. Then $p_{k+1}(z)=\overline{b}_0p_k-b_{n-k}p_k^\#$ and $p_{k+1}(0)=|b_0|^2-|b_{n-k}|^2>0$ so $|b_0|>|b_{n-k}|$. Since $p_k(z)\neq 0$ on $\mathbf{T}$ we have $|\overline{b}_0p_k|>|b_{n-k}p_k|=|b_{n-k}p_k^\#|$ on $\T$ and thus Rouché's theorem implies that $p_{k+1}$ and $p_k$ have equally many zeros in $\overline{\mathbf{D}}$.
\end{proof}

We are now ready to formulate Schur-Cohn's test (e.g. \cite{Henrici74}).

\begin{theorem}\label{thm: SchurCohn}
A polynomial $p\in\mathcal{P}_n$ has no zeros in $\overline{\mathbf{D}}$ if and only if $p_k(0)>0$ for $k=1,\ldots,n.$
\end{theorem}

\begin{proof}
Assume that $p=p_0$ has no zeros in $\overline{\mathbf{D}}$. By Lemma \ref{lemma: SC}, $p_1(0)=|a_0|^2-|a_n|^2>0$. Since $p_0$ has no zeros on $\mathbf{T}$, Lemma
\ref{SchurLemma} implies that $p_1$ and $p_0$ has the same number of zeros in $\overline{\mathbf{D}}$, i.e., none. If $n\ge 2$ we can repeat the reasoning with $p_1\in\mathcal{P}_{n-1}$ instead of $p_0$ and deduce $p_2(0)>0$ and so on, and after a finite number of steps we finally get $p_n(0)>0$.
	
	To prove the reverse implication, assume that $p_k(0)>0$ for $k=1,\ldots,n$. Then $p_n$ is a nonzero constant, and especially has no zeros on $\mathbf{T}$. Since $p_n=S_1p_{n-1}$, $p_{n-1}$ has no zeros on $\mathbf{T}$ and by Lemma \ref{SchurLemma}, $p_{n-1}$ and $p_n$ have the same number of zeros in $|z|\le 1$, i.e., none. We may now repeat this with $p_{n-1}\in\mathcal{P}_1$ instead of $p_n$ and deduce that $p_{n-2}$ has no zeros in $\overline{\mathbf{D}}$ and so on, and after a finite number of steps we finally get that $p=p_0$ has no zeros in $\overline{\mathbf{D}}$.
\end{proof}

\section{Real Comprehensive Triangular Decomposition} \label{sec:background}
The method of real triangular decompositions, introduced recently in \cite[\S4]{chen2011semi}, \cite[\S10]{chenThesis} and \cite{chen2007comprehensive},
provides a suitable framework to deal with the
the mapping problem for quadrature domains, in the form of systems such as \eqref{meq2}, \eqref{con2}. These methods in turn make use of the idea of a Cylindrical Algebraic Decomposition \cite[\S5]{basu2006algorithms}, for an overview of this and other applicable methods from real algebraic geometry see, for example, the book \cite{basu2006algorithms}.

The objects we consider here are {\em semi-algebraic sets}, given a list of polynomial equations and inequalities in $\RR[x_1,\dots, x_m]$ {\em a basic semi-algebraic set} is the set of all points in $\RR^m$ which simultaneously satisfy all these equations and inequalities, \cite[\S3]{basu2006algorithms}. In this section we will specifically consider semi-algebraic systems defined by polynomials in the polynomial ring $$\QQ[c,x]=\QQ[c_1,\dots, c_d][x_1,\dots, x_n],$$
where we think of the $c_i$ as parameters and the $x_j$ as variables.

More precisely, given polynomials $f_1,\dots, f_r$ and $p_1,\dots, p_s$ in $\QQ[c,x]$, we define the {\em semi-algebraic system} $[f_{=0},p_{>0}]$ to be the following set of equations and inequalities
\begin{equation}\label{SAS}f_1(c,x)=0,\quad \cdots\quad f_r(c,x)=0,\quad p_1(c,x)>0,\quad \cdots\quad p_s(c,x)>0.\end{equation}
The set of real solutions $(c,x)\in \R^d\times\R^n$ to \eqref{SAS} is called the
{\em (parameterized) semi-algebraic set} generated by the system, denoted by
$\mathcal{S}([f_{=0},p_{>0}])$.
Moreover, for fixed $c\in \RR^d$ we
define the {\em specialized semi-algebraic set} $
\mathcal{S}_{(c)}([f_{=0},p_{>0}])$
as the set of points $x\in\RR^n$ which satisfy the system \eqref{SAS} for the particular parameter-value $c$.

Now suppose that $\mathfrak{T}=\left\lbrace \mathcal{T}_1,\dots, \mathcal{T}_\ell \right\rbrace$ is a collection of semi-algebraic systems in $\QQ[c,x]$. We extend the definitions of semi-algebraic set and specialization to a parameter value $c$ by $$\mathcal{S}(\mathfrak{T})=\bigcup_{j=1}^\ell\mathcal{S}(\mathcal{T}_j)\subset \RR^d\times \RR^n, \qquad \mathcal{S}_{(c)}(\mathfrak{T})=\bigcup_{j=1}^\ell\mathcal{S}_{(c)}(\mathcal{T}_j)\subset  \RR^n.$$

Moreover, given a semi-algebraic system $\mathcal{T}=[f_{=0},p_{>0}]$ we define the \textit{constructible set} of $\mathcal{T}$ to be the set of \textit{complex} solutions
$(c,x)\in\C^d\times \C^n$ to the system of equations and inequalities
\begin{equation}\label{CON}f_1(c,x)=\cdots=f_r(c,x)=0,\quad p_1(c,x)\neq 0,\quad \cdots,\quad p_s(c,x)\neq 0.\end{equation}
We denote by $\CC\mathcal{S}(\mathcal{T})$ the set of solutions $(c,x)\in \C^d\times\C^n$ to \eqref{CON}; given
$c\in \CC^d$ we define the associated specialization $\CC\mathcal{S}_{(c)}(\mathcal{T})$ to be the set of points $x\in\CC^n$ such that $(c,x)\in \CC\mathcal{S}(\mathcal{T})$.

If $\mathfrak{T}=\left\lbrace \mathcal{T}_1,\dots, T_\ell \right\rbrace$ is a finite collection of semi-algebraic systems, it should now be obvious how to extend our definitions of constructible set ($\CC\mathcal{S}(\mathfrak{T})=\cup_{j} \CC\mathcal{S}(\mathcal{T}_j)$) and specializations
($\CC\mathcal{S}_{(c)}(\mathfrak{T})=\cup_{j}\CC\mathcal{S}_{(c)}(\mathcal{T}_j)\subset \CC^n$).

A semi-algebraic system $ \mathcal{T}=[f_{=0},p_{>0}]$ is called {\em square-free} if all polynomials $f_j$ and $p_i$ occurring in $\mathcal{T}$ are square-free.
(A polynomial $g\in \QQ[c,x]$ is square-free if it has no factor of the form
$w^2$ where $w\in \QQ[c,x]$ is non-constant.)

\begin{definition} Let $\mathcal{W}=[f_{=0},p_{>0}]$ be a semi-algebraic system defined by polynomials $f_1,\dots, f_r$ and $p_1,\dots, p_s$ in $\QQ[c,x]$ and let $\mathcal{S}(\mathcal{W})\subset \RR^d\times \RR^n$ be the associated semi-algebraic set.

A {\em real comprehensive triangular decomposition (RCTD)}
of $\calW$ is a pair $(\mathfrak{C},\left\lbrace \mathfrak{T}_C \;;\; C\in \mathfrak{C}\right\rbrace )$ where
$\frakC$ is a finite partition
of $\R^d$ into non-empty semi-algebraic sets $C$ (called ``cells'') and for each $C\in \mathfrak{C}$, $\mathfrak{T}_C$ is a finite set of square-free semi-algebraic systems such that exactly one of the following holds:
        \begin{enumerate}
        \item $\mathfrak{T}_C$ is empty so $\mathcal{S}(\mathfrak{T}_C)=\RR^d\times \RR^n$ and $\mathcal{S}_{(c)}(\mathfrak{T}_C)=\RR^n$ for all $c\in \RR^d$,
        \item
        The specialized constructible set $\CC\mathcal{S}_{(c)}(\mathfrak{T}_C)$ is infinite for all $c\in C$,
        \item $\mathfrak{T}_C=\left\lbrace \mathcal{T}_1,\dots, \mathcal{T}_\ell\right\rbrace$ is a finite set of semi-algebraic systems satisfying the following conditions:
        \begin{itemize}
            \item $\CC\mathcal{S}_{(c)}(\mathfrak{T}_C)$ is finite and has fixed cardinality for all $c\in C$,
            \item the specialized semi-algebraic sets $\mathcal{S}_{(c)}(\mathcal{T}_j)$ are finite and non-empty for all $j$ and further for a fixed $\mathcal{T}_j$ the specialized semi-algebraic set $\mathcal{S}_{(c)}(\mathcal{T}_j)$ has fixed cardinality for all $c\in C$,
            \item $\mathcal{S}_{(c)}(\mathcal{W})=\bigsqcup_{j=1}^\ell\mathcal{S}_{(c)}(\mathcal{T}_j) $ for all $c\in C$.
        \end{itemize}
    \end{enumerate}\label{def:RCTD}
\end{definition}

The following proposition summarizes the results about RCTD's that we will apply in the sequel.

\begin{proposition}[\S10 of \cite{chenThesis}]
Let $f_1,\dots, f_r$ and $p_1,\dots, p_s$  be polynomials in the polynomial ring $\QQ[c,x]$
defining a semi-algebraic system $\mathcal{W}=[f_{=0},p_{>0}]$. Then a real comprehensive triangular decomposition 
of $\mathcal{W}$
exists and may be computed by an explicit algorithm which is guaranteed to terminate in finite time. This algorithm is implemented in the \texttt{RegularChains} Maple package \cite{RegularChains}.
\end{proposition}

\begin{example}
\label{ex:RCTD}

Define $f=a\cdot y^2+b\cdot x +3$ and $g=b\cdot xy-y+2$ where $x,y$ denote real variables and $a,b$ denote real parameters.
Consider the semi-algebraic system $$
\mathcal{W}=[f=0,g=0,y\leq 1].
$$

Using the \texttt{RegularChains} package, the parameter space
$\RR^2$ is partitioned into five cells $\mathfrak{C}=\{C_0,C_1,C_2,C_3, C_\infty\}$ where \begin{align*}
&C_\infty=\left\lbrace \left( -\frac{3}{4},0\right) \right\rbrace,\;\;\;C_0=\left\lbrace (a,b)\in \RR^2 \;;\; a\neq -\frac{3}{4},\;\; b=0 \right\rbrace, \\
&C_1=\left\lbrace(a,b)\in \RR^2 \;;\; a<-\frac{64}{27}, \;b\neq 0 \right\rbrace \bigsqcup \left\lbrace(a,b)\in \RR^2 \;|\; a\geq 0,\; b\neq 0 \right\rbrace, \\
&C_2=\left\lbrace(a,b)\in \RR^2 \;;\; a=-\frac{64}{27}, \;b\neq 0\right\rbrace\bigsqcup \left\lbrace(a,b)\in \RR^2 \;|\; -2<a<0, \; b\neq 0 \right\rbrace,\\
&C_3=\left\lbrace(a,b)\in \RR^2 \;;\; -\frac{64}{27}<a\leq -2, \; b\neq 0 \right\rbrace.
\end{align*} These cells are illustrated in Figure \ref{fig:RCTD_Example}.
\begin{figure}[h]
\includegraphics[scale=0.3]{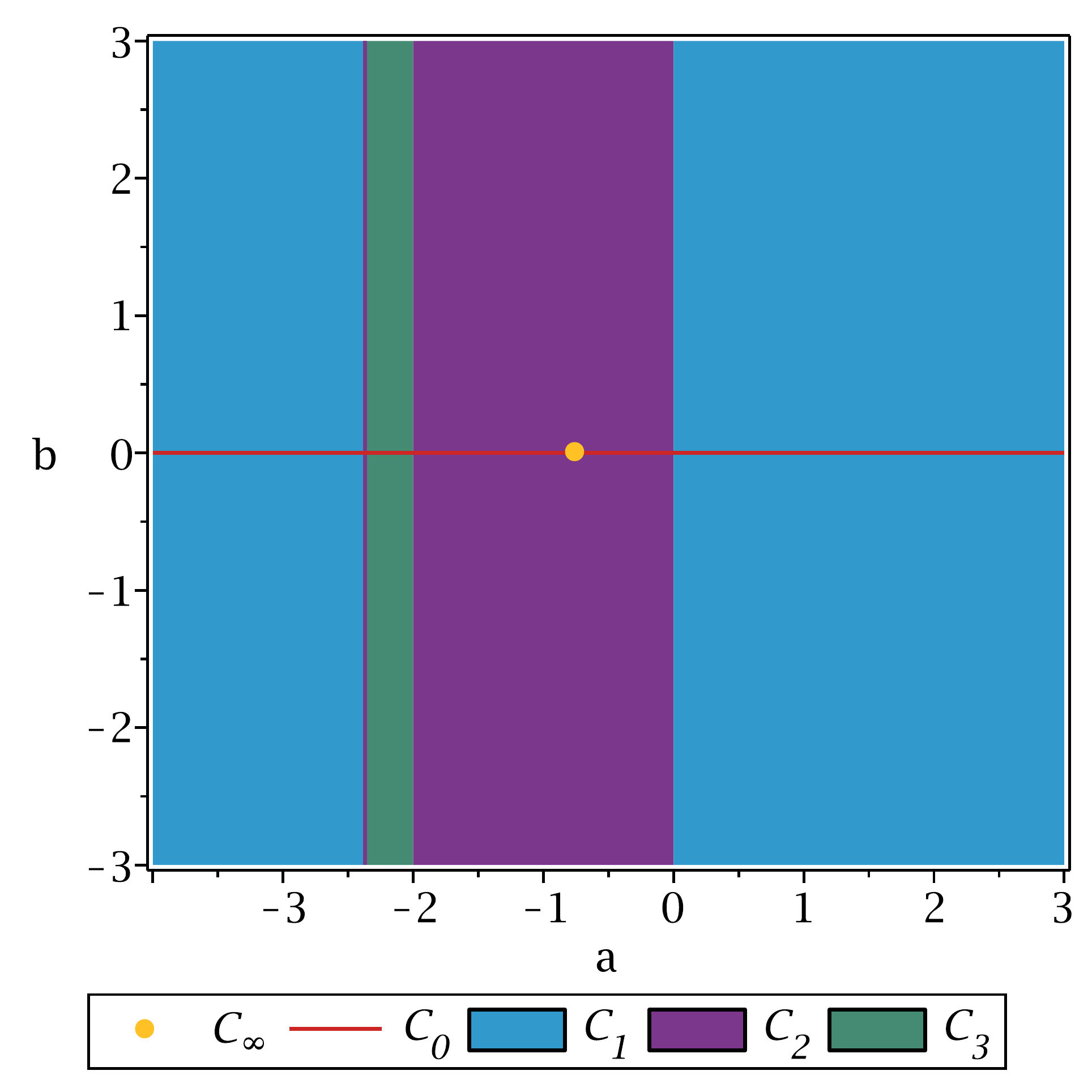}
\caption{The cells $\mathfrak{C}$ from the RCTD in Example \ref{ex:RCTD}. \label{fig:RCTD_Example}}

\end{figure}

It is not hard to show that
the specialized semi-algebraic set $\mathcal{S}_{(a,b)}(\mathcal{W})$ consists of $j$ points for all $(a,b)\in C_j$ where $j\in \{0,1,2,3\}$; for the cell $C_\infty$ the RCTD guarantees only that the specialized constructible set associated to the parameter choice $C_\infty$ has infinitely many points (a corresponding specialized semi-algebraic set may be infinite, empty, or finite). Hence $\mathfrak{C}$, along with the associated semi-algebraic systems for each parameter cell, gives a RCTD of $\mathcal{W}$.

In more detail; consider the cells $C_0$ and $C_\infty$. The disjoint union of these cells is the line $b=0$ (i.e.~the a-axis in Figure \ref{fig:RCTD_Example}). Along this line the semi-algebraic system $\mathcal{W}$ simplifies to
\begin{equation}
a\cdot y^2+3=2-y=0, \;\; y\leq 1.\label{eq:simpSysEx}
\end{equation}
 Solving for $y$ we obtain $y=\sqrt{\frac{-3}{a}}$ and $y=2$, so $f=g=0$ has a real solution if and only if $a=\frac{-3}{4}$, 
 so $\{(a,b)\}=C_\infty$. The corresponding specialized constructible set is $$\left\lbrace (x,y)\in \CC^2\;;\;(y-2)(y+2)=(y+2)=0, \; \; y\neq 1\right\rbrace=\{(x,2)\;;\; x\in \CC\},$$
 which is infinite. However since $y\leq 1$ is never satisfied the semi-algebraic set $\calS(\calI_{C_\infty})$ is empty.
On the other hand in the cell $C_0$ it is clear that the semi-algebraic system \eqref{eq:simpSysEx} has no real solutions, since $a\ne\frac{-3}{4}$ when $(a,b)\in C_0$.

Now consider the subset $P$ of the cell $C_2$ consisting of all $(a,b)$ where $a=-\frac{64}{27}$ and $b\neq 0$. For $(a,b)\in P$ the system $\mathcal{W}$ simplifies to
\begin{equation}
\left[x=\frac{7}{3b}, \;\; y=\frac{-3}{2},\;\; y\leq 1\right]\;\; {\rm or}\;\;\left[x=\frac{-5}{3b}, \;\; y=\frac{3}{4},\;\; y\leq 1\right]. \label{eq:simpTwoSol}
\end{equation}
Clearly the system \eqref{eq:simpTwoSol} has precisely two solutions for any $b\neq 0$, namely $(x,y)=\left(\frac{7}{3b}, \frac{-3}{2} \right)$ and  $(x,y)=\left(\frac{-5}{3b}, \frac{3}{4}\right)$. Similarly all other choices of $(a,b)\in C_2$ yield associated specialized semi-algebraic sets with exactly two points.
\end{example}

\section{Some Computational Results}
In this section we apply the method of real comprehensive triangular decompositions to obtain (new) proofs of uniqueness for certain families of quadrature domains.

\subsection{Aharonov-Shapiro 1976}\label{section:AS}
In this subsection, we shall give an alternative proof of a theorem of Aharonov and Shapiro from \cite{Aharonov1976}, which states that a solid quadrature domain obeying a quadrature identity of the form
\begin{equation}\label{eq: AS 1976}
\int_\Omega f\ dA=M_0f(0)+M_1f'(0),\qquad f\in AL^1(\Omega)
\end{equation}
is unique.
Here $M_0$ is the area of $\Omega$, so necessarily $M_0>0$. The constant $M_1$ is allowed to be an arbitrary complex number.

Recall that the computations in Example \ref{Example: monopole and dipole} show that a normalized mapping function $\fii:\D\to\Omega$ is necessarily a polynomial of degree at most $2$, which solves the system
\begin{equation}\label{casas}
\begin{cases}
w_1^3\bar{w}_1=M_0w_1^2-M_1w_2\\
w_1^2\bar{w}_2=2M_1
\end{cases},\qquad (w_1:=\fii'(0)>0,\quad w_2:=\fii''(0)).
\end{equation}
It remains to show that this system gives rise to a unique univalent solution.
For this, we write
$$M_1:=m_1+in_1,\qquad w_2:=u_2+iv_2.$$

We now obtain the following semi-algebraic system, which is equivalent to \eqref{casas},
\begin{equation*}
(\mathcal{W})\qquad
\begin{cases}
w_1^4=M_0w_1^2-(m_1u_2-n_1v_2)\\
0=m_1v_2+n_1u_2\\
w_1^2u_2=2m_1\\
-w_1^2v_2=2n_1\\
M_0>0\\
w_1>0\\
(w_1^2M_0-(m_1u_2-n_1v_2))^2-4w_1^2(m_1^2+n_1^2)\ge 0
\end{cases}.\label{eq:AS_Sys}
\end{equation*}

We must verify that the system $\mathcal{W}$ gives rise to at most one univalent solution for all relevant choices of quadrature data.
For this, we first recognize that the last condition in $\mathcal{W}$ is just the Schur-Cohn constraint (see Theorem \ref{thm: SchurCohn}), which ensures that $\fii'(z)\neq 0$ in $\mathbf{D}$. In general this is only necessary for univalence but for polynomials of degree two it is also sufficient.
We shall treat $c=(m_1,n_1,M_0)$
 as parameters.
Computing a real comprehensive triangular decomposition
 of the system $\mathcal{W}$, using the \texttt{RegularChains} library,
 we obtain a partition of the parameter space $\RR^3$ into two cells $C_0,C_1$ having the following properties.

All points in the cell $C_0$
 are such that $M_0\leq 0$; hence no point of $C_0$ can correspond to a quadrature domain.
 With $\omega:=\sqrt[3]{(27m_1^2+27n_1^2)/2}$,
 the cell $C_1$
is expressed as the disjoint union of the following six subsets of $\R^3$,
\begin{enumerate}[label=(\roman*)]
	\item \label{i1} \begin{equation*}
	\begin{cases}
	M_0=\omega\\
	m_1=m_1\\
	n_1<0
	\end{cases},\ \
	\begin{cases}
	M_0=\omega\\
	m_1=m_1\\
	n_1>0
	\end{cases}\label{eq:firstRegionC1}
	\end{equation*}
	\item \label{i2} \begin{equation*}
	\begin{cases}
	M_0=\omega\\
	m_1<0\\
	n_1=0
	\end{cases}
	\end{equation*}
	\item \label{i3} \begin{equation*}
	\begin{cases}
	M_0=\omega\\
	m_1>0\\
	n_1=0
	\end{cases}
	\end{equation*}
	\item \label{i4} \begin{equation*}
	\begin{cases}
	M_0>\omega\\
	m_1=m_1\\
	n_1<0
	\end{cases},\ \
	\begin{cases}
	M_0>\omega\\
	m_1=m_1\\
	n_1>0
	\end{cases}
	\end{equation*}
	\item \label{i5} \begin{equation*}
	\begin{cases}
	M_0>\omega\\
	m_1<0\\
	n_1=0
	\end{cases},\ \
	\begin{cases}
	M_0>\omega\\
	m_1>0\\
	n_1=0
	\end{cases}
	\end{equation*}
	\item \label{i6} \begin{equation*}
	\begin{cases}
	M_0>0\\m_1=0\\n_1=0
	\end{cases}.\label{eq:lastRegionC1}
	\end{equation*}
\end{enumerate}

For $(M_0,m_1,n_1)$ in each of these sets, we now prove that the system \eqref{casas} has a unique solution $(w_1,w_2)=(w_1,u_2+iv_2)$ with $w_1>0$.
Indeed, straightforward calculations show that in each of the parameter-domains,
\ref{i1}--\ref{i6}
the semi-algebraic system $\mathcal{W}$ simplifies to, respectively,
\begin{enumerate}[label=(\roman*)]
	\item \begin{equation*}
	\begin{cases}
	v_2w_1^2+2n_1=0\\
	n_1u_2+v_2m_1=0\\
	(9m_1^2+9n_1^2)v_2+2M_0^2n_1=0\\
	w_1>0
	\end{cases}\label{eq:firstSimp}
	\end{equation*}
	\item \begin{equation*}
	\begin{cases}
	2u_2w_1-u_2^2+2M_0=0\\
	9m_1u_2-2M_0^2=0\\
	v_2=0
	\end{cases}
	\end{equation*}
	\item \begin{equation*}
	\begin{cases}
	2u_2w_1+u_2^2-2M_0=0\\
	9m_1u_2-2M_0^2=0\\
	v_2=0
	\end{cases}
	\end{equation*}
	\item \begin{equation*}
	\begin{cases}
	v_2w_1^2+2n_1=0\\
	n_1u_2+v_2m_1=0\\
	(m_1^2+n_1^2)v_3^2-2M_0n_1^2v_2-4n_1^3=0\\
	w_1>0\\
	-v_{{2}}^{2}M_{{0}}m_{{1}}^{2}n_{{1}}^{2}-v_{{2}}^{2}M_{{0}}
	n_{{1}}^{4}+2\,M_{{0}}^{2}n_1^{4}+6\,v_{{2}}m_{{1}}^{2}n_
		{{1}}^{3}+6\,v_{{2}}n_{{1}}^{5}>0
	\end{cases}
	\end{equation*}
	\item \begin{equation*}
	\begin{cases}
	u_2w_1^2-2m_1=0\\
	u_2^3-2M_0u_2+4m_1=0\\
	v_2=0\\
	w_1>0\\
	-M_{{0}}m_{{1}}^{2}u_{{2}}^{2}+2\,M_{{0}}^{2}m_{{1}}^{2}-6\,
	m_{{1}}^{3}u_{{2}}-4\,m_{{1}}n_{{1}}^{2}u_{{2}}>0
	\end{cases}
	\end{equation*}
	\item \begin{equation*}
	\begin{cases}
	w_1^2-M_0=0\\
	u_2=0\\
	v_2=0\\
	w_1>0
	\end{cases}
	\end{equation*}
\end{enumerate}

In each case \ref{i1}--\ref{i6}, we have a unique solution $(w_1,w_2)=(w_1,u_2+iv_2)$ where $w_1>0$, which concludes our automated proof of uniqueness for
solid q.d.'s obeying \eqref{eq: AS 1976}. $\qed$

\begin{remark} The uniqueness problem for quadrature domains of order $2$ has been completely settled, see \cite[Corollary 10.1]{G83}.
More precisely, counting also non simply connected domains and combining with results in \cite{Gu88}, a given point functional $\mu$ of order $2$ can give rise
to at most two quadrature domains: one simply connected $\Omega$, and possibly another one of the form $\Omega\setminus\{a\}$ where $a$ is a ``special point'',
i.e., $S(a)=\bar{a}$ where $S$ is the Schwarz function. The term ``special point'' is due to Shapiro in \cite{Sh87}, see Theorem 2.9 and the discussion that precedes it.

These methods can also be applied to polynomials of higher degree. (This topic is being considered as part of an ongoing work by the two of the authors.)
\end{remark}

\subsection{Symmetric smash sums}
\label{smashs}
Consider a domain $\Omega$ that satisfies
\begin{equation}\label{smash}\int_\Omega f\, dA=c(f(a_0)+\cdots+f(a_{n-1})),\qquad a_j=a e^{2\pi i j/n},\qquad f\in AL^1(\Omega),\end{equation}
where $c>0$ and $a<0$ are constants and $n$ is a positive integer.

A domain satisfying \eqref{smash} may be constructed as a potential theoretic sum (or ``smash sum'') of discs $D(a_j;\sqrt{c})$, where excess mass coming from overlapping discs
is swept out using the process of partial balayage. When $n=2$ this process gives rise to the well-known Neumann's oval.

 For $n\ge 3$ one can surmise that a (connected) domain satisfying \eqref{smash} should be either simply connected or doubly connected, depending on whether or not $0\in \Omega$. The doubly connected case has already been treated in \cite[Section 5.8]{Varchenko1992} and in \cite{Crowdy02}. (See Subsection \ref{mcq} for further remarks in this connection.)

In this subsection, we shall focus on the simply connected case and show how uniqueness of the quadrature domain follows using our general algebraic scheme.
(In this particular case, there are, of course, alternative ways to see this, e.g.~by Novikov's theorem, Theorem \ref{settles}, \ref{no38}.)

 Due to the $\mathbb{Z}_n$ symmetry of \eqref{smash}, the conformal map $\phi:\mathbf{D}\to\Omega$ attains the simple form
\begin{equation}\label{eq: Z_n mapping}
\phi(z)=\sum_{j=0}^{n-1}\frac{c_j}{\bar{w}_j}\frac{z}{1-\bar{\lambda}_jz}=\frac{c}{w}\sum_{j=0}^{n-1}\frac{z}{1-\exp\left(\frac{-2\pi i}{n}j\right)\lambda z}=\frac{cn}{w}\frac{z}{1-\lambda^nz^n}
\end{equation}
where $w_j=\phi'(\lambda_j)=w>0$, $-1<\lambda=\phi^{-1}(a)<0$.
This gives
\begin{equation}\label{eq: Z_n phi_prime}
\phi'(z)=\frac{cn}{w}\frac{1+(n-1)\lambda^nz^n}{(1-\lambda^nz^n)^2}.
\end{equation}
The two unknowns, $\lambda$ and $w$ ($-1<\lambda<0$ and $w>0$) satisfy
\begin{equation*}\label{eq: Z_n alg system}
\begin{cases}\displaystyle
a=\frac{cn}{w}\frac{\lambda}{1-\lambda^{2n}}\\
\displaystyle
w=\frac{cn}{w}\frac{1+(n-1)\lambda^{2n}}{(1-\lambda^{2n})^2}
\end{cases}
\qquad\iff\qquad
\begin{cases}
aw(1-\lambda^{2n})=nc\lambda\\
w^2(\lambda^{4n}-2\lambda^{2n}+1)=nc((n-1)\lambda^{2n}+1).
\end{cases}
\end{equation*}

We only need to check solutions $\fii$ which are locally univalent.
From \eqref{eq: Z_n phi_prime} (or from Schur-Cohn's test) one may easily deduce that local univalence holds if $1\ge (n-1)\lambda^{n}$.
Taking this into consideration we obtain the following
semi-algebraic system
\begin{align}
\label{SAS7}
\begin{cases}
aw(1-\lambda^{2n})=nc\lambda\\
w^2(\lambda^{4n}-2\lambda^{2n}+1)=nc((n-1)\lambda^{2n}+1)\\
c>0\\
a<0\\
w>0\\
-1<\lambda<0\\
1\ge(n-1)^2\lambda^{2n}.
\end{cases}
\end{align}

 When setting up a real comprehensive triangular decomposition, we cannot treat $n$ as a parameter but have to use a fixed value of $n$ (since if $n$ is treated as a parameter the system \eqref{SAS7} is no longer semi-algebraic). Below we show the results for $n=3$ but the corresponding decompositions may be obtained for up to $n=10$ without any runtime issues.

Thus let $\mathcal{W}$ be the semi-algebraic system \eqref{SAS7} with $n=3$ fixed. The associated semi-algebraic set $\mathcal{S}(\mathcal{W})$ is contained in $\RR^2\times \RR^2$. We treat $(a,c)\in \RR^2$ as parameters and
consider the specialized semi-algebraic set $\mathcal{S}_{(a,c)}(\mathcal{W})\subset \RR^2$.
Using an RCTD, we have found that the parameter space $\RR^2$  is divided into two cells $C_0,C_1$
such that for all $(a,c)\in C_1$ the system $\mathcal{W}$ has precisely one real solution and for all $(a,c)\in C_0$ the system $\mathcal{W}$ has no real solutions.
The cell $C_1$ is made up of the three regions defined in \ref{I1}, \ref{I1}, and \ref{I3} below.
\begin{enumerate}[label=(\roman*)]
	\item \label{I1} \begin{equation*}
				\begin{cases}\displaystyle
				c=\frac{a^2}{2^{1/3}}\\
				a<0
				\end{cases}\label{eq:Zn1}
				\end{equation*}
	\item \label{I2} \begin{equation*}
	\begin{cases}\displaystyle
	\frac{a^2}{2^{1/3}}<c<a^2\\
	a<0
	\end{cases}, \ \ \ \begin{cases}
	a^2<c\\
	a<0
	\end{cases}\label{eq:Zn2}
	\end{equation*}
	\item \label{I3} \begin{equation*}
	\begin{cases}
	c=a^2\\
	a<0
	\end{cases}.\label{eq:Zn3}
	\end{equation*}
\end{enumerate}
In each of the domains above, the semi-algebraic system \eqref{SAS7} simplifies as follows
\begin{enumerate}[label=(\roman*)]
	\item \begin{equation*}
	\begin{cases}
	(\lambda^{6}-1)aw+3\lambda c=0\\
	a^2\lambda+c=0
	\end{cases}\label{eq:Zn1Sys}
	\end{equation*}
	\item \begin{equation*}
	\begin{cases}
	(\lambda^{6}-1)aw+3\lambda c=0\\
	2a^2\lambda^6-3c\lambda^2+a^2=0\\
	w>0\\
	-4\lambda^6+1>0\\
	-\lambda>0\\
	\lambda+1>0
	\end{cases}
	\end{equation*}
	\item  \begin{equation*}
	\begin{cases}
	(\lambda^2-1)w+2\lambda a=0\\
	2\lambda^4+2\lambda^2-1=0\\
	w>0\\
	-4\lambda^6+1>0\\
	-\lambda>0\\
	\lambda+1>0
	\end{cases}
	\end{equation*}
\end{enumerate}

Since for all choices of the parameters in the cell $C_1$ we know from the definition of the RCTD that the number of real solutions is fixed, picking a particular $a$ and $c$ in $C_1$ it is easy to verify that the systems above have unique real solutions which correspond to a quadrature domain. Similarly, one can pick a choice of parameters in $C_0$ to verify that the system \eqref{SAS7} has no real solutions for any parameter choice in $C_0$. In particular, we have shown that a quadrature domain obeying a quadrature identity \eqref{smash} is uniquely determined.
 However, we cannot yet be sure that each solution to the system \eqref{SAS7}
 defines a univalent map since we have only required local univalence. The following proposition shows that this is in fact the case.

\begin{proposition}
The mapping $\phi$ in \eqref{eq: Z_n mapping}, ($-1<\lambda<0$), is univalent if and only if
	\begin{equation}\label{66}
	1\ge (n-1)^2\lambda^{2n}.
	\end{equation}
\end{proposition}
\begin{proof}
	The necessity is obvious since the strict inequality $1>(n-1)^2\lambda^{2n}$ is just the Schur-Cohn condition for $\phi'$ to have no zeros in $\overline{\mathbf{D}}$.
	
	To prove sufficiency we shall prove that \eqref{66} implies that $\phi$ is starlike, and in particular univalent. Recall (\cite[Theorem 2.5]{P75}) that starlikeness is equivalent to that
	$	\mathrm{Re}\left(z\frac{\phi'(z)}{\phi(z)}\right)>0$ for all $z\in\D$.
	Using the Cauchy-Riemann equations one finds that this is equivalent to
	\begin{equation*}
	\frac{d}{dr}|\phi(z)|^2>0,\qquad z=re^{i\theta}\in\mathbf{D}.
	\end{equation*}
(Cf. \cite{Fejer1951}).
Using Eq. (\ref{eq: Z_n mapping}) we find that $\frac{d}{dr}|\phi(z)|^2$ is a positive multiple of
	\begin{equation}\label{pos}
	2r\left(\lambda^{2n}r^{2n}(1-n)-\lambda^nr^n(2-n)\cos(n\theta)+1\right).
	\end{equation}
This is a quadratic polynomial in $\lambda^n$ which is clearly positive when $\lambda=0$. We need to prove that it remains positive for all $r$ with $0\le r< 1$ if $\lambda<0$ and \eqref{66} holds.
We may assume that $r=1$ and $\cos(n\theta)=\pm 1$.
For odd $n$, the non-negativity of \eqref{pos} then means
	\begin{equation}\label{in1}
	\lambda^{2n}(1-n)-\lambda^n(2-n)+1\ge 0
	\end{equation}
	while for even $n$, it means
	\begin{equation}\label{in2}
	\lambda^{2n}(1-n)+\lambda^n(2-n)+1\ge 0.
	\end{equation}
	By assumption $\lambda<0$, so the inequalities \eqref{in1}, \eqref{in2} amount to  $\lambda^n(n-1)\le 1$ and $-1\ge \lambda^n(n-1)$ respectively. This is equivalent to $1\ge \lambda^{2n}(n-1)^2$ and proves that $\phi$ is starlike if and only if \eqref{66} holds.
\end{proof}

\section{Three examples of uniqueness}
In this section we give several further examples of uniqueness of quadrature domains based on real comprehensive triangular decomposition and other computational methods from algebraic geometry (i.e.~Gröbner basis computations).

\begin{example}
Let $\Omega$ be a quadrature domain satisfying
\begin{equation}\label{Eq: QI example I}
\int_\Omega f\ dA=f(0)+f\left(\frac{265}{153}+i\right)+f\left(\frac{265}{153}-i\right),\qquad f\in AL^1(\Omega)
\end{equation}
the nodes are chosen such that the union $D(0,1)\cup D(265/153+i,1)\cup D(265/153-i,1)$ is just barely simply connected (this is guaranteed by $265/153<\sqrt{3}$), see Fig. \ref{fig: non-starlike example I}.

To find the mapping $\phi:\mathbf{D}\to\Omega$ we need to solve the six complex equations $a_i=\phi(\lambda_i),\ w_i=\phi'(\lambda_i),\ i=1,2,3$, or equivalently 12 real equations. Without loss of generality we pick $\lambda_1=0=a_0$ and $w_1>0$, moreover, using the mirror symmetry in \eqref{Eq: QI example I} we get $\lambda_2=\bar{\lambda}_3$ and $w_2=\bar{w}_3$. This means that in total we only have to solve for five real variables. The equations are of course rational, meaning that in order to calculate the Gröbner basis we need to clear denominators and algebraically remove the new roots introduced by this. These new roots  correspond to poles of the original rational system, and hence are not of interest. The defining equations without roots at the poles are obtained by saturating the resulting ideal by the product of the equations from the denominators.

A Gröbner basis calculation shows that there is a unique solid quadrature domain obeying \eqref{Eq: QI example I}, depicted in Fig. \ref{fig: non-starlike example I}.
To be explicit, we can write $(\lambda_1,\lambda_2,\lambda_3)=(0,l+ip,l-ip)$ and
 $(w_1,w_2,w_3)=(x_1,x_2+iy_2,x_2-iy_2)$, where
\begin{equation*}\label{eq: parameters example I}
\begin{array}{c}
l=0.866022320861578...,\ p=0.499994659802733...,\ x_1=1.00001067993828...,\\ x_2=93633.9999439279...,\ y_2=0.866068567766777....
\end{array}
\end{equation*}

\begin{figure}[tbh!]
	\centering
	\includegraphics[width=0.7\textwidth]{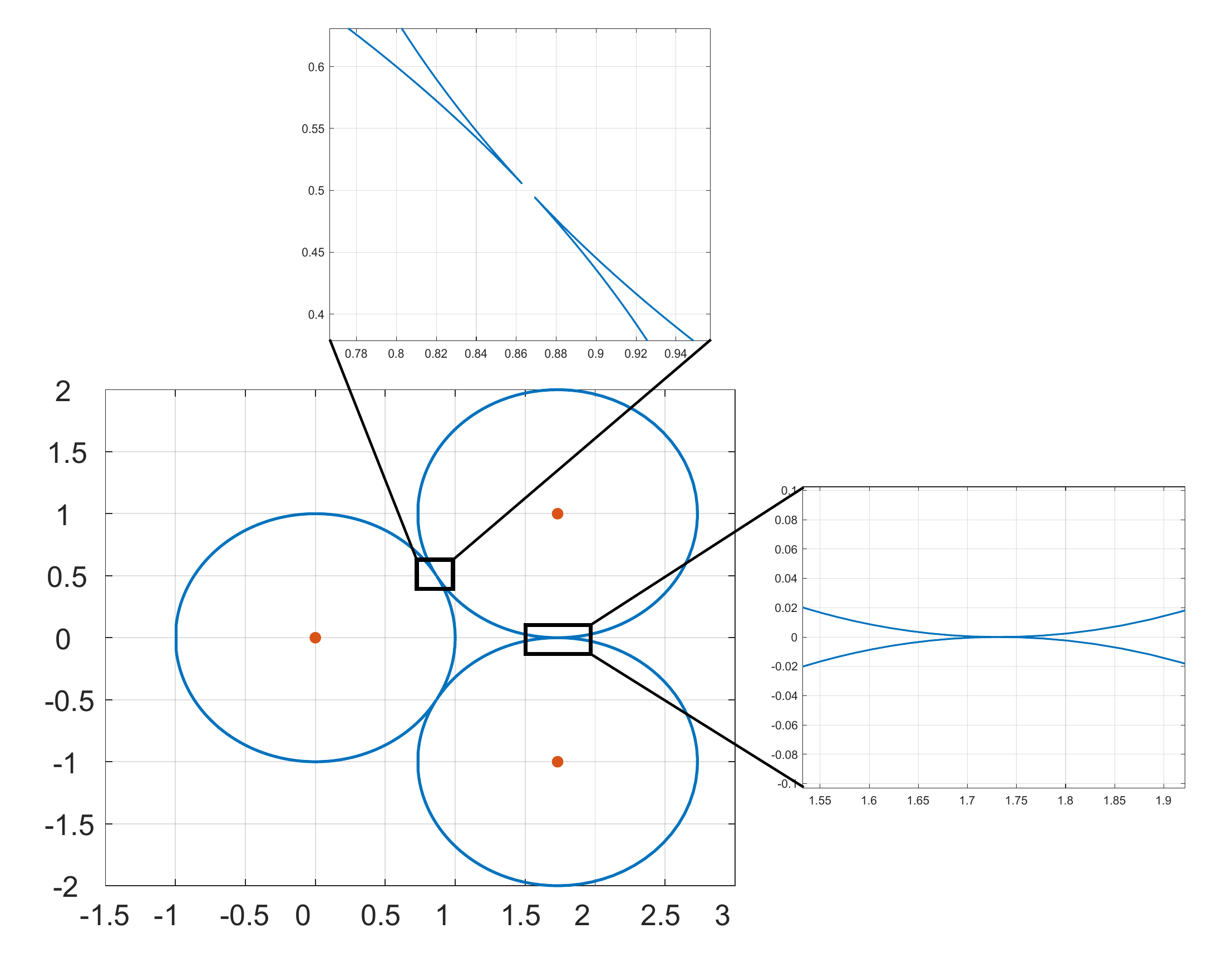}
	\caption{The quadrature domain defined by Eq. \eqref{Eq: QI example I}. (A close inspection shows that the domain is in fact simply connected.)}
	\label{fig: non-starlike example I}
\end{figure}
Note that the values $l,p,x_1,x_2,y_2$ above can be obtained as exact expressions in a field extension of $\QQ$ from the result of our Gröbner basis computation, however we opt to give the numerical approximations here for simplicity.
\end{example}

\begin{example}
Let $c_1>0$ and let $\Omega$ be a quadrature domain satisfying
\begin{equation*}
\int_\Omega f\ dA=c_1\left[f\left(-2\right)+f\left(-2e^{2\pi i/3}\right)+f\left(-2e^{4\pi i/3}\right)\right]+f(0),\qquad f\in AL^1(\Omega).
\end{equation*}
This is similar to a
$\mathbb{Z}_3$-symmetric domain in Subsection \ref{smashs},
but with an additional node at the origin.
Assuming $\Omega$ is simply connected, the mapping $\phi:\mathbf{D}\to\Omega$ is given by
\begin{equation*}\label{eq: map example II}
\phi(z)=\frac{3c_1}{w_1}\frac{z}{1-\lambda^3z^3}+\frac{c_2}{w_2}z
\end{equation*}
where $\phi(\lambda)=-2=a,\ \phi'(\lambda)=w_1,\ \phi'(0)=w_2$ and $c_2=1$.
Clearing denominators leads to the semi-algebraic system
\begin{equation}
\begin{cases}
3c_1\lambda(\lambda^6-1-w_2)=-2w_1w_2(\lambda^6-1)\\
3c_1(\lambda^12-2\lambda^6(1-w_2)+1+w_2)=w_1^2w_2(\lambda^6-1)\\
3c_1(1+w_2)=w_1w_2^2\\
c_1>0\\
-1<\lambda<0\\
w_1>0,\;w_2>0
\end{cases}\label{eq:ExIIS5}
\end{equation}with variables $w_1,w_2,\lambda$ and parameter $c_1$.
\begin{figure}[h]
    \centering
    \begin{tabular}{c @{\hspace{1.1\tabcolsep}} c @{\hspace{1.1\tabcolsep}} c}
      \includegraphics[trim={70 620 430 10},clip,width=0.3\textwidth]{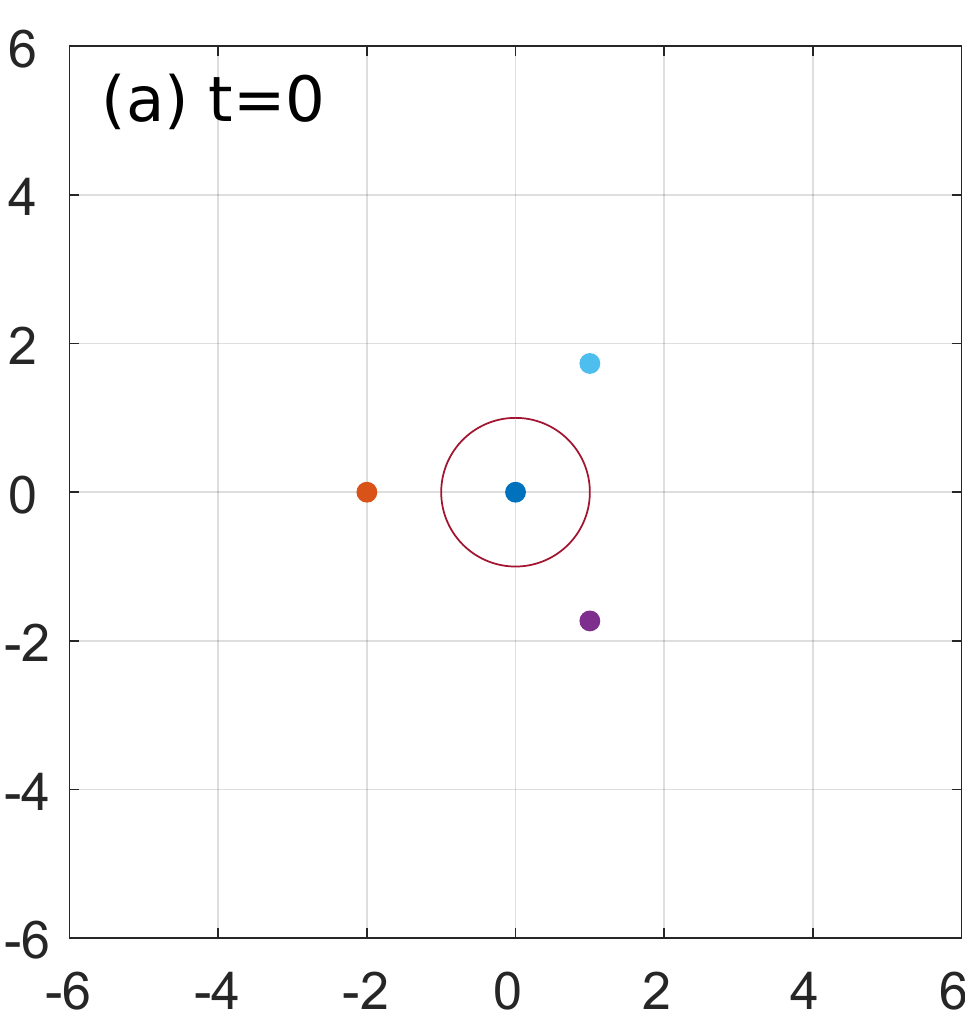} \hspace{-37mm}&
    \includegraphics[trim={423 620 73 10},clip,width=0.305\textwidth]{HS.pdf}&
    \includegraphics[trim={70 320 430 310},clip,width=0.299\textwidth]{HS.pdf}
    \\
          \includegraphics[trim={423 320 73 321},clip,width=0.302\textwidth]{HS.pdf} &
 \includegraphics[trim={70 10 430 617},clip,width=0.295\textwidth]{HS.pdf}&
 \includegraphics[trim={423 10 73 610},clip,width=0.295\textwidth]{HS.pdf}
    \end{tabular}
	\caption{Hele-Shaw growth of the unit disc $\D$ under injection at constant rate $1$ at three nodes $-2,-2e^{2\pi i/3},-2e^{4\pi i/3}$. (a) and (b) lie in the cell $C_0$ of the RCTD; (a) is the initial domain $\D$ (time $t=0$) with the three injection points plotted, (b) corresponds to $t=1/2$ and is disconnected. In (c), $t=c_1>1$ so the domain is simply connected and lies in cell $C_1$ as do (d), (e), and (f).}
	\label{fig: HS}
\end{figure}
Let $\mathcal{W}$ denote the semi-algebraic system \eqref{eq:ExIIS5}. Computing a triangular decomposition yields two cells: $C_0=\{ c_1\in \RR \;|\; c_1\leq 1 \}$ and $C_1=\{ c_1\in \RR \;|\; c_1> 1 \}$ in the parameter space $\RR$ (with coordinate $c_1$). For each $c_1\in C_0$ the system $\mathcal{W}$ has no real solutions while for each $c_1\in C_1$ the system $\mathcal{W}$ has exactly one real solution $(w_1,w_2,\lambda)\in \RR^3$; note that $\RR=C_0 \sqcup C_1$.

Observe that to obtain a solution, we must have $c_1>1$ in order that $\Omega$ be simply connected (see Fig. \ref{fig: HS}(c)).
If we interpret $c_1$ as time, we may think of the quadrature domain $\Omega(c_1)$ as the result of a Hele-Shaw evolution, starting from
$\Omega(0)=D(0,1)$, after injecting fluid through the nodes $-2,-2e^{2\pi i/3},-2e^{4\pi i/3}$ at constant rate $1$, for $c_1$ units of time (see Subsection \ref{growth} below for more about this). The evolution is depicted in Fig. \ref{fig: HS}.
\end{example}

\begin{example}
Let $\Omega$ be a quadrature domain satisfying
\begin{equation}\label{eq: QI example III}
\int_\Omega f\ dA=-\frac{1}{2}(f(a)+f(-a))+2f(0),\qquad f\in AL^1(\Omega).
\end{equation}
This is a $\mathbb{Z}_n$- symmetric domain with $n=2$ as in Subsection \ref{smashs} but with negative weights and an additional node at the origin with strength two.
The mapping $\phi:\mathbf{D}\to\Omega$ then satisfies
\begin{equation*}
\phi(z)=\frac{2c_1}{w_1}\frac{z}{1-\lambda^2z^2}+\frac{c_2}{w_2}z
\end{equation*}
where $\phi(\lambda)=a,w_1=\phi'(\lambda),w_2=\phi'(0),c_1=-1/2$ and $c_2=2$. To obtain unique solutions for this system we have to impose Schur-Cohn constraints; this is as easily done for arbitrary $n$ as for $n=2$ so we do it for arbitrary $n$.

Set $p(z)=b_0+b_1z^{n}+b_2z^{2n}$ where $b_0=c_1w_2n,\ b_1=\lambda^n(c_1w_2n(n-1)-2c_2w_2)$ and $b_2=c_2w_1\lambda^{2n}$. The derivative of the mapping $\phi$ is given by
\begin{equation*}
\phi'(z)=\frac{c_1n}{w_1}\frac{1+(n-1)\lambda^nz^n}{(1-\lambda^nz^n)^2}+\frac{c_2}{w_2}=\frac{p(z)}{w_1w_2(1-\lambda^nz^n)^2}.
\end{equation*}
When $n=2$ we have that $b_0=4c_1w_2,\ b_1=\lambda^2(2c_1w_2-2c_2w_2)$ and $b_2=c_2w_1\lambda^{4}$. It follows that the Schur-Cohn constraints for the $n=2$ case are given by
\begin{equation*}
\begin{cases}
b_0^2-b_2^2\ge 0\\
(b_0^2-b_2^2)^2-(b_0b_1-b_1b_2)^2\ge 0,
\end{cases}
\end{equation*}
where we allow for equality to include the case where $\d \Omega$ has cusps, see Fig. \ref{fig: exIII}.
\begin{figure}[t]
	\centering
	\includegraphics[trim=4cm 9.5cm 4cm 9.5cm,clip,width=0.5\textwidth]{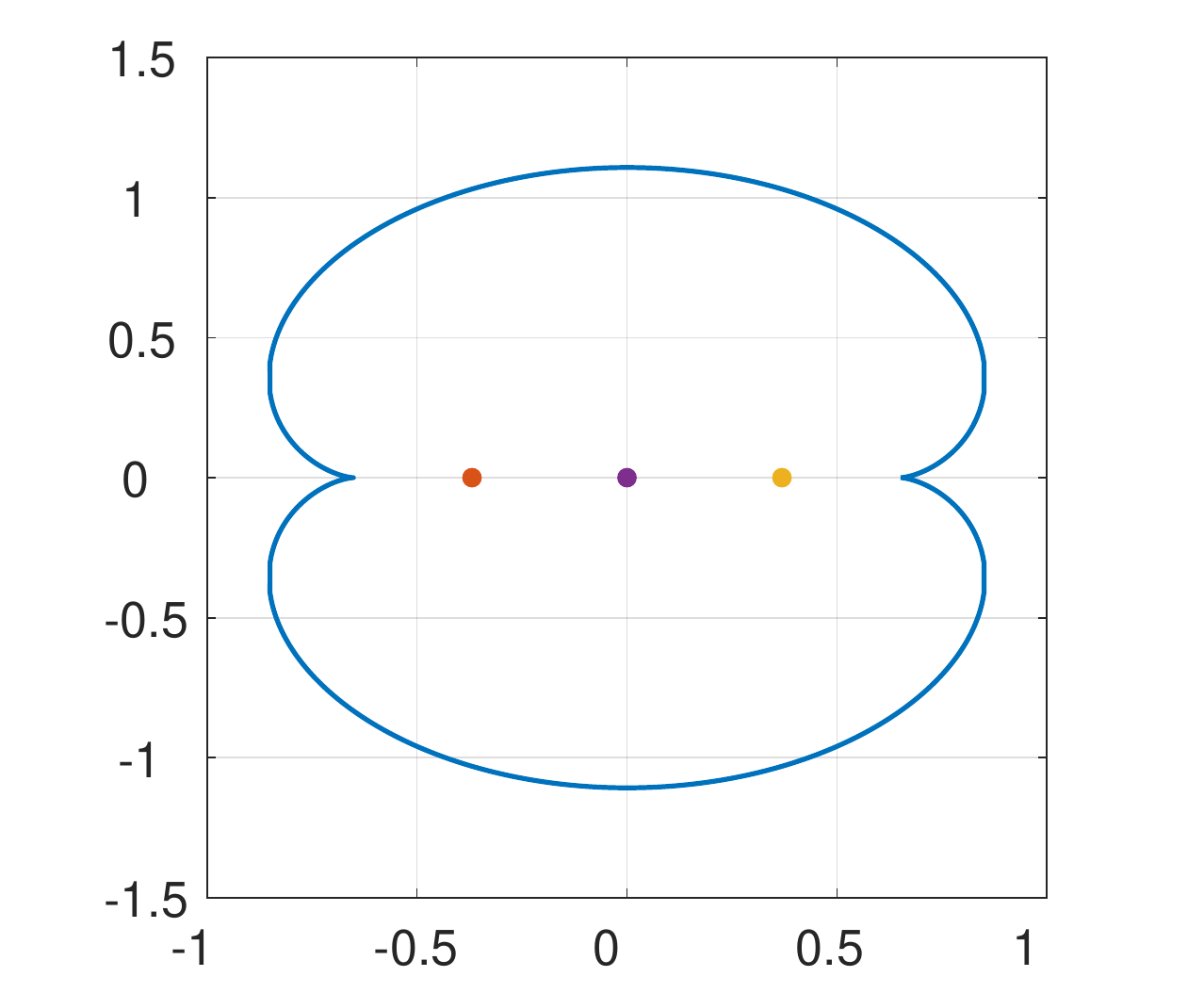}
	\caption{The unique solid quadrature domain satisfying \eqref{eq: QI example III} for the parameter-value $a=\vartheta \in C_1$.}
	\label{fig: exIII}
\end{figure}
Clearing denominators in the equations $\phi(\lambda)=a,w_1=\phi'(\lambda),w_2=\phi'(0)$, recalling that $c_1=-1/2$, $c_2=2$, and adding the Schur-Cohn constraints we obtain the semi-algebraic system
\begin{equation}
\begin{cases}
\lambda\, \left( 2\,{\lambda}^{4}{w_1}-2\,{ w_1}+{ w_2}
\right) =aw_{{1}}w_{{2}} \left( {\lambda}^{4}-1 \right) \\
w_1^2w_2\left( {\lambda}^{4}-1 \right) ^{2}=\phi'(\lambda)= \left( 2\,{\lambda}^{8}-4\,{\lambda}^{4}+2 \right) w_{{1}}-{\lambda}^{4}w_{{2}}-w_{{2}}\\
w_1w_2^2=2w_1-w_2\\
b_0^2-b_2^2\ge 0\\
(b_0^2-b_2^2)^2-(b_0b_1-b_1b_2)^2\ge 0\\
a<0\\
-1<\lambda<0\\
w_1,w_2>0
\end{cases}\label{eq:EX3}
\end{equation}
where we treat $a$ as a parameter and $(w_1,w_2,\lambda)$ as variables (note $c_1,c_2$ are fixed). Let $\mathcal{W}$ denote the semi-algebraic system \eqref{eq:EX3}. Let $\tau=27\sqrt{3}+54$ and set $$\vartheta=         	
-\frac{1}{2} \, \sqrt{\frac{{\tau}^{\frac{2}{3}} + 4 \, {\tau}^{\frac{1}{3}} + 9}{{\tau}^{\frac{1}{3}}}} + \frac{1}{2} \, \sqrt{-{\tau}^{\frac{1}{3}} - \frac{9}{{\tau}^{\frac{1}{3}}} + \frac{16}{\sqrt{\frac{{\tau}^{\frac{2}{3}} + 4 \, {\tau}^{\frac{1}{3}} + 9}{{\tau}^{\frac{1}{3}}}}} + 8}\cong -0.369.
$$\normalsize

 A computation of a real comprehensive triangular decomposition for $\mathcal{W}$ gives that the parameter space $\RR$ (with coordinate $a$) is partitioned into two cells: $C_0=\{a\in \RR\;|\; a<\vartheta \} \sqcup \{a\in \RR\;|\; a\geq 0 \}$ and $C_1=\{a\in \RR\;|\; \vartheta \leq a <0 \}$. For all $a\in C_0$ the system $\mathcal{W}$ has no solutions and for all $a\in C_1$ the system $\mathcal{W}$ has exactly one real solution $(w_1,w_2,\lambda)\in \RR^3$. It follows that for any valid choice of the parameter $a$ there exists a unique quadrature domain. The parameter value $a=\vartheta\in C_1$ corresponds to the case where $\d \Omega$ has two cusps, as shown in Fig. \ref{fig: exIII}.
\end{example}

\section{Other related topics} \label{concrem}

In this section, we briefly discuss some of the existing methods and constructions in the theory of quadrature domains with bearing for our above methods. The reader interested in a comprehensive picture of the area should consult one or several of the textbooks \cite{Crowdy2020,Davis74,Gustafsson2006,Gustafsson2014,Shapiro1992,Varchenko1992}.


\subsection{The defining polynomial of a q.d.} \label{mupp} Consider a quadrature domain $\Omega$ of order $n=n_1+\cdots+n_m$ corresponding to a point-functional
\begin{equation}\label{pf2}\mu(f)=\sum_{k=1}^m\sum_{j=0}^{n_k-1} c_{kj}f^{(j)}(a_k).\end{equation}
It was shown by Aharonov and Shapiro in \cite{Aharonov1976} that the boundary $\d\Omega$ is an algebraic curve, and more precisely there exists
an irreducible polynomial $P(z,w)=\sum_{j,k=0}^n a_{jk}z^jw^k$ which is self-conjugate and normalized (i.e. $a_{jk}=\bar{a}_{kj}$, $a_{nn}=1$) such that
\begin{equation}\label{special}\d\Omega=\{z\,;\, P(z,\bar{z})=0\}\setminus \{\text{``special points''}\},\end{equation}
where a ``special point''  is an isolated solution to the equation $P(z,\bar{z})=0$. (Cf.~\cite{Sh87})

It is convenient to refer to the polynomial $P$ in \eqref{special} as the ``defining polynomial'' for the quadrature domain~$\Omega$.
A natural problem, then, is to try to characterize the set of defining polynomials $P$ which are associated with a given point-functional $\mu$. This problem has been
investigated by Gustafsson in \cite{G83,Gu88}, and leads to interesting insights with respect to the uniqueness question (Q).
To elucidate this, we write a defining polynomial in the form
\begin{equation}\label{alg0}P(z,w)=\sum_{j=0}^n w^jp_j(z),\end{equation}
where each $p_j(z)$ is a polynomial of degree at most $n$.

It is shown in \cite{G83} (cf.~ also \cite{Crowdy2001}) that there is an explicit bijection between point-functionals $\mu$ (of the form \eqref{pf2}) and the last
two polynomials $p_{n-1},p_n$ which may appear in the expansion \eqref{alg0}. As pointed out in \cite{G83,Crowdy2001}, the determination of the remaining polynomials $p_0,\ldots,p_{n-2}$
is generally a difficult matter. In the simply connected case, this problem is similar to the question of completely characterizing all solutions to the master formula (in Theorem
\ref{theorem: structure of univalent function}), for each point-functional $\mu$.

\begin{remark}
The notion of a special point 
is not just an artefact of the construction. 
For example, such points appear naturally in the process of forming smash sums of discs (i.e., each time discs ``collide'' as in Figure \ref{fig: HS} (c))
 and they have the physical interpretation of stagnation points for fluid flows.
We refer to \cite{Gu88,Crowdy2001} 
for further details.
\end{remark}

\subsection{Connection to Laplacian growth} \label{growth}

Let $t$ be a real parameter and consider a family of quadrature domains $\Omega(t)$, each obeying the quadrature identity
\begin{equation}\label{evol}\int_{\Omega(t)} f\, dA=\mu_t(f)=(c_1+qt)f(a_1)+c_2f(a_2)+\ldots+c_nf(a_n),\qquad f\in AL^1(\Omega),\end{equation}
where $q$ is a real constant, and where we take $a_1=0$ for simplicity.
If $q>0$, we may think of $\Omega(t)$ as an expanding blob of fluid, obtained from $\Omega(0)$ by injecting fluid at the origin, at constant rate $q$.
(If $q<0$, the domains contract due to suction.)
The resulting evolution $t\mapsto \Omega(t)$ is known under the names ``Hele-Shaw evolution'' and ``Laplacian growth'', see e.g. \cite{Mineev2009,Gustafsson2014,Varchenko1992}.

For values of $t$ such that $\Omega(t)$ is solid, we will write $z\mapsto \fii(z,t)$ for the Riemann mapping $\D\to \Omega(t)$.

 \begin{remark} Any domain $\Omega\in Q(\mu_t,AL^1)$ is by definition of finite order, and hence its boundary is part of an algebraic curve, see Subsection \ref{mupp}. This implies
that the boundary curve is smooth everywhere with the possible exception of finitely many singular points, which may be either cusps $p\in\d\Omega$ pointing inwards (corresponding to values $p=\fii(z)$, $z\in\T$ at which $\fii'(z)=0$), or contact points $p\in\d\Omega$
(which satisfy $p=\fii(z_1)=\fii(z_2)$ for two distinct points $z_1,z_2\in\T$). For solid domains, contact points are excluded.
\end{remark}

In the following, we suppose that $\Omega(t)\in Q(\mu_t,AL^1)$ is a smoothly varying family of \textit{solid} domains obeying \eqref{evol} for $t$ in some suitable time-interval.
A basic result relates
the ``time-derivative'' $\dot{\fii}(z,t):=\frac \d {\d t} \fii(z,t)$ to the ``space-derivative'' $\fii'(z,t)=\frac \d {\d z}\fii(z,t)$.

\begin{theorem} \label{PG}
The conformal map $\fii(z,t)$ onto $\Omega(t)$ satisfies Polubarinova-Galin's equation
\begin{equation}\label{pgeq}\re\left[\dot{\fii}(z,t)\overline{z\fii'(z,t)}\right]=\frac q 2,\qquad z\in\T.\end{equation}
\end{theorem}

A proof can be found in the book \cite{Gustafsson2006}, Section 1.4.2.

\medskip

Gustafsson and Lin in \cite{GL13} have studied the evolution of zeros and poles of the space-derivative $\fii'(z,t)$. We will now briefly
indicate a different possible approach based on our basic structure theorem, Theorem \ref{theorem: structure of univalent function}.
Denote by $\lambda_i(t)$ the points in $\D$ such that $\fii(\lambda_i(t),t)=a_i$, and write $w_i(t)=\fii'(\lambda_i(t),t)$.
Also denote $c_1(t)=c_1+qt$ and $c_j(t)\equiv c_j$ when $j\ge 2$.
In view of \eqref{meq}, the mapping $\fii(z,t)$ obeys
\begin{equation}\label{meq3}\fii(z,t)=\sum_{i=1}^n \frac {\overline{c_i(t)}}{\overline{w_i(t)}}\frac {z}{1-\overline{\lambda_i(t)} z},\end{equation}
which gives
\begin{equation}\label{meq4}\fii'(z,t)=\sum_{j=1}^n\frac {\overline{c_i(t)}}{\overline{w_i(t)}}\frac {1}{(1-\overline{\lambda_i(t)} z)^2}\end{equation}
and (denoting complex conjugation by $a^\dagger=\bar{a}$, and abbreviating $c_i=c_i(t)$, $\lambda_i=\lambda_i(t)$, etc.)
\begin{equation}\label{meq5}\dot{\fii}(z,t)=z\left[\frac 1 {w_1}-\sum_{i=1}^n(\frac {c_i\dot{w}_i}
{w_i^2})^\dagger\frac 1 {1-z\bar{\lambda}_i}+\sum_{i=1}^n\frac {\bar{c}_i}{\bar{w}_i}\frac {z\overline{\dot{\lambda_i}}}{(1-z\bar{\lambda}_i)^2}\right].\end{equation}

Substituting \eqref{meq4} and \eqref{meq5} in Polubarinova-Galin's equation \eqref{pgeq}, we obtain a nonlinear system of ordinary differential equations connecting the
functions $\lambda_i(t)$ and $w_i(t)$ and their time-derivatives $\dot{\lambda}_i(t),\dot{w}_i(t)$. We may note from \eqref{meq3} that the poles of the Riemann map $\fii(z,t)$ are given by $1/\overline{\lambda_i(t)}$.

Similar sets of equations have been studied also in the case when $\fii(z,t)=a_1(t)z+\ldots +a_n(t)z^n$ is a polynomial in $z$, see \cite[Section 2.1.1]{Gustafsson2014} and  \cite{Mineev1990}. It is clear that our method applies in this case as well, but for reasons of length, we shall not pursue this issue here. 




\subsection{Multi-connected quadrature domains} \label{mcq} A theory for multiply connected quadrature domains, using the Schottky double of the domain, is developed in the paper \cite{G83}.
The mapping problem for such domains has been the subject of numerous investigations.
In particular, for a significant class of quadrature domains (e.g.~based on forming suitable smash sums of discs), the Riemann map can be constructed using the corresponding ``Schottky-Klein prime function''.
This approach is found in the works \cite{Crowdy2004,Crowdy2005} as well as
in the
recent monograph \cite[Chapter 11]{Crowdy2020}. 
(The Schottky-Klein function is surveyed in the article \cite{Crowdy2010}.)

Other relevant works in this connection are \cite{LM} on the topology of quadrature domains, and \cite{GS19}, which discusses related uniqueness questions.





\subsection*{Acknowledgments}
This work was supported in part by the Anders Wall Foundation.

\bibliographystyle{plain}
\bibliography{refs}

\begin{thebibliography}{10}

\bibitem{Mineev2009}
Ar. Abanov, M.~Mineev-Weinstein, and A.~Zabrodin.
\newblock Multi-cut solutions of {L}aplacian growth.
\newblock {\em Phys. D}, 238(17):1787--1796, 2009.

\bibitem{Aharonov1976}
Dov Aharonov and Harold~S. Shapiro.
\newblock Domains on which analytic functions satisfy quadrature identities.
\newblock {\em Journal d'Analyse Math{\'e}matique}, 30(1):39--73, Dec 1976.

\bibitem{basu2006algorithms}
Saugata Basu, Richard Pollack, and Marie-Fran{\c{c}}oise Roy.
\newblock {\em Algorithms in real algebraic geometry}, volume~10.
\newblock Springer, 2006.

\bibitem{B65}
Lipman Bers.
\newblock An approximation theorem.
\newblock {\em J. Analyse Math.}, 14:1--4, 1965.

\bibitem{chenThesis}
Changbo Chen.
\newblock {\em Solving polynomial systems via triangular decomposition}.
\newblock PhD thesis, The University of Western Ontario, 2011.

\bibitem{chen2007comprehensive}
Changbo Chen, Oleg Golubitsky, Fran{\c{c}}ois Lemaire, Marc~Moreno Maza, and
  Wei Pan.
\newblock Comprehensive triangular decomposition.
\newblock In {\em International Workshop on Computer Algebra in Scientific
  Computing}, pages 73--101. Springer, 2007.

\bibitem{chen2011semi}
Changbo Chen and Marc~Moreno Maza.
\newblock Semi-algebraic description of the equilibria of dynamical systems.
\newblock In {\em International Workshop on Computer Algebra in Scientific
  Computing}, pages 101--125. Springer, 2011.

\bibitem{Crowdy2001}
Darren Crowdy.
\newblock Multipolar vortices and algebraic curves.
\newblock {\em R. Soc. Lond. Proc. Ser. A Math. Phys. Eng. Sci.},
  457(2014):2337--2359, 2001.

\bibitem{Crowdy02}
Darren Crowdy.
\newblock The construction of exact multipolar equilibria of the
  two-dimensional {E}uler equations.
\newblock {\em Phys. Fluids}, 14(1):257--267, 2002.

\bibitem{Crowdy2005}
Darren Crowdy.
\newblock Quadrature domains and fluid dynamics.
\newblock In Peter Ebenfelt, Bj{\"o}rn Gustafsson, Dmitry Khavinson, and Mihai
  Putinar, editors, {\em Quadrature Domains and Their Applications}, pages
  113--129, Basel, 2005. Birkh{\"a}user Basel.

\bibitem{Crowdy2010}
Darren Crowdy.
\newblock The {S}chottky-{K}lein prime function on the {S}chottky double of
  planar domains.
\newblock {\em Comput. Methods Funct. Theory}, 10(2):501--517, 2010.

\bibitem{Crowdy2020}
Darren Crowdy.
\newblock {\em Solving Problems in Multiply Connected Domains}.
\newblock NSF-CBMS Regional Conference Series in Applied Mathematics: 97. SIAM,
  Philadelphia USA, 2020.

\bibitem{Crowdy2004}
Darren Crowdy and Jonathan Marshall.
\newblock Constructing multiply connected quadrature domains.
\newblock {\em SIAM J. Appl. Math.}, 64(4):1334--1359, 2004.

\bibitem{D72}
Philip~J. Davis.
\newblock Double integrals expressed as single integrals or interpolatory
  functionals.
\newblock {\em J. Approximation Theory}, 5:276--307, 1972.

\bibitem{Davis74}
Philip~J. Davis.
\newblock {\em The {S}chwarz function and its applications}.
\newblock The Mathematical Association of America, Buffalo, N. Y., 1974.
\newblock The Carus Mathematical Monographs, No. 17.

\bibitem{Fejer1951}
{Fej\'{e}r, L. and Szeg\"{o}, G.}
\newblock Special conformal mappings.
\newblock {\em Duke Math. J.}, 18(2):535--548, 06 1951.

\bibitem{Gardiner2008}
Stephen~J. Gardiner and Tomas Sj\"{o}din.
\newblock {Convexity and the Exterior Inverse Problem of Potential Theory}.
\newblock {\em Proceedings of the American Mathematical Society},
  136(5):1699--1703, 2008.

\bibitem{GS19}
Stephen~J. Gardiner and Tomas Sj\"{o}din.
\newblock A characterization of annular domains by quadrature identities.
\newblock {\em Bull. Lond. Math. Soc.}, 51(3):436--442, 2019.

\bibitem{G83}
Bj\"{o}rn Gustafsson.
\newblock Quadrature identities and the {S}chottky double.
\newblock {\em Acta Appl. Math.}, 1(3):209--240, 1983.

\bibitem{Gu88}
Bj\"{o}rn Gustafsson.
\newblock Singular and special points on quadrature domains from an algebraic
  geometric point of view.
\newblock {\em J. Analyse Math.}, 51:91--117, 1988.

\bibitem{G90}
Bj\"{o}rn Gustafsson.
\newblock On quadrature domains and an inverse problem in potential theory.
\newblock {\em J. Analyse Math.}, 55:172--216, 1990.

\bibitem{GL13}
Bj\"{o}rn Gustafsson and Yu-Lin Lin.
\newblock On the dynamics of roots and poles for solutions of the
  {P}olubarinova-{G}alin equation.
\newblock {\em Ann. Acad. Sci. Fenn. Math.}, 38(1):259--286, 2013.

\bibitem{GLR}
Bj\"{o}rn Gustafsson, Yu-Lin Lin, and Joakim Roos.
\newblock Laplacian growth on branched riemann surfaces.
\newblock To Appear.

\bibitem{GuS05}
Bj\"{o}rn Gustafsson and Harold~S. Shapiro.
\newblock What is a quadrature domain?
\newblock In {\em Quadrature domains and their applications}, volume 156 of
  {\em Oper. Theory Adv. Appl.}, pages 1--25. Birkh\"{a}user, Basel, 2005.

\bibitem{Gustafsson2006}
Bj\"{o}rn Gustafsson and Alexander Vasil'ev.
\newblock {\em {Conformal and Potential Analysis in Hele-Shaw Cell}}.
\newblock Advances in Mathematical Fluid Mechanics. Birkh\"{a}user Verlag,
  2006.

\bibitem{Gustafsson2014}
Bj\"{o}rn Gustafsson, Alexander Vasil'ev, and Razvan Teodorescu.
\newblock {\em Classical and stochastic Laplacian growth.}
\newblock Advances in Mathematical Fluid Mechanics. Springer, 2014.

\bibitem{G07}
Björn Gustafsson and Mihai Putinar.
\newblock Selected topics on quadrature domains.
\newblock {\em Physica D: Nonlinear Phenomena}, 235(1):90 -- 100, 2007.
\newblock Physics and Mathematics of Growing Interfaces.

\bibitem{Henrici74}
Peter Henrici.
\newblock {\em Applied and computational complex analysis. {V}ol. 1}.
\newblock Wiley Classics Library. John Wiley \& Sons, Inc., New York, 1988.
\newblock Power series---integration---conformal mapping---location of zeros,
  Reprint of the 1974 original, A Wiley-Interscience Publication.

\bibitem{Isakov}
Victor Isakov.
\newblock {\em Inverse source problems}.
\newblock Mathematical surveys and monographs: 34. American Mathematical
  Society, 1990.

\bibitem{LM}
Seung-Yeop Lee and Nikolai~G. Makarov.
\newblock Topology of quadrature domains.
\newblock {\em J. Amer. Math. Soc.}, 29(2):333--369, 2016.

\bibitem{RegularChains}
F.~Lemaire, M.~Moreno~Maza, and Y.~Xie.
\newblock {The RegularChains library}.
\newblock In I.~Kotsireas, editor, {\em Proceedings of Maple Conference}, pages
  355--368. Maplesoft, 2005.

\bibitem{Mineev1990}
M.~B. Mineev.
\newblock A finite polynomial solution of the two-dimensional interface
  dynamics.
\newblock {\em Phys. D}, 43(2-3):288--292, 1990.

\bibitem{Novikoff1938}
Par~P Novikoff.
\newblock Sur le problème inverse du potentiel.
\newblock {\em {Comptes Rendus (Doklady) de l'Acad\'{e}mie des Sciences de
  l'URSS}}, 18(3):165--168, 1938.

\bibitem{P75}
Christian Pommerenke.
\newblock {\em Univalent functions}.
\newblock Vandenhoeck \& Ruprecht, G\"{o}ttingen, 1975.
\newblock With a chapter on quadratic differentials by Gerd Jensen, Studia
  Mathematica/Mathematische Lehrb\"{u}cher, Band XXV.

\bibitem{Sakai1999}
M.~Sakai.
\newblock Linear combinations of harmonic measures and quadrature domains of
  signed measures with small supports.
\newblock In {\em {Proceedings of the Edinburgh Mathematical Society}},
  volume~42, pages 433--444, 1999.

\bibitem{Sakai1982}
Makoto Sakai.
\newblock {\em Quadrature Domains}.
\newblock Springer Verlag, 1982.

\bibitem{SA88}
Makoto Sakai.
\newblock Finiteness of the family of simply connected quadrature domains.
\newblock In {\em Potential theory ({P}rague, 1987)}, pages 295--305. Plenum,
  New York, 1988.

\bibitem{Sakai1998}
Makoto Sakai.
\newblock {Sharp Estimates of the Distance from a Fixed Point to the Frontier
  of a Hele-Shaw Flow}.
\newblock {\em Potential Analysis}, 8(3):277--302, May 1998.

\bibitem{Sh87}
Harold~S. Shapiro.
\newblock Unbounded quadrature domains.
\newblock In {\em Complex analysis, {I} ({C}ollege {P}ark, {M}d., 1985--86)},
  volume 1275 of {\em Lecture Notes in Math.}, pages 287--331. Springer,
  Berlin, 1987.

\bibitem{Shapiro1992}
Harold~S. Shapiro.
\newblock {\em {The Schwarz function and its generalization to higher
  dimensions}}, volume~9 of {\em The University of Arkansas lecture notes in
  the mathematical sciences}.
\newblock Wiley, 1992.

\bibitem{Skinner}
Brian Skinner.
\newblock {\em Logarithmic {P}otential {T}heory on {R}iemann {S}urfaces}.
\newblock ProQuest LLC, Ann Arbor, MI, 2015.
\newblock Thesis (Ph.D.)--California Institute of Technology.

\bibitem{Ullemar1980}
Carina Ullemar.
\newblock A uniqueness theorem for domains satisfying a quadrature identity for
  analytic functions.
\newblock Technical Report TRITA-MAT-1980-37, KTH Royal Institute of
  Technology, 1980.

\bibitem{Varchenko1992}
A.~N. Var\u{c}enko and P.~I. Etingof.
\newblock {\em Why the boundary of a round drop becomes a curve of order four.}
\newblock University lecture series: 3. American Mathematical Society, 1992.

\bibitem{Za87}
Lawrence Zalcman.
\newblock Some inverse problems of potential theory.
\newblock In {\em Integral geometry ({B}runswick, {M}aine, 1984)}, volume~63 of
  {\em Contemp. Math.}, pages 337--350. Amer. Math. Soc., Providence, RI, 1987.

\bibitem{zidarov1973method}
D~Zidarov.
\newblock Method of finding point (dipole) solutions of the potential field
  inverse problem.
\newblock {\em pure and applied geophysics}, 110(1):1918--1926, 1973.

\bibitem{zidarov1970obtaining}
D~Zidarov and Zh~Zhelev.
\newblock On obtaining a family of bodies with identical exterior fields-method
  of bubbling.
\newblock {\em Geophysical Prospecting}, 18(1):14--33, 1970.

\end{thebibliography}

\end{document}